\RequirePackage[l2tabu, orthodox]{nag}

\documentclass[12pt,reqno]{amsart}
\usepackage{fullpage,url,amssymb,enumerate,colonequals}
\usepackage{mathrsfs} % for \mathscr (script letters)

\usepackage{pdfsync}

\usepackage[usenames,dvipsnames]{xcolor}

\usepackage[OT2,T1]{fontenc}
\DeclareSymbolFont{cyrletters}{OT2}{wncyr}{m}{n}
\DeclareMathSymbol{\Sha}{\mathalpha}{cyrletters}{"58}

\usepackage{color}

\newcommand{\defi}[1]{\textsf{#1}} % for defined terms

\newcommand{\Aff}{\mathbb{A}}
\newcommand{\C}{\mathbb{C}}

\newcommand{\F}{\mathbb{F}}

\newcommand{\Q}{\mathbb{Q}}
\newcommand{\R}{\mathbb{R}}
\newcommand{\Z}{\mathbb{Z}}

\newcommand{\nnn}{\eta}

\newcommand{\abs}[1]{\lvert {#1} \rvert}

\newcommand{\calA}{\mathcal{A}}

\newcommand{\calD}{\mathcal{D}}
\newcommand{\calE}{\mathcal{E}}

\newcommand{\calG}{\mathcal{G}}
\newcommand{\calH}{\mathcal{H}}

\newcommand{\calO}{\mathcal{O}}

\newcommand{\calV}{\mathcal{V}}

\renewcommand{\AA}{\mathscr{A}}
\newcommand{\DD}{\mathscr{D}}
\newcommand{\EE}{\mathscr{E}}

\newcommand{\KK}{\mathscr{K}}

\newcommand{\TT}{\mathscr{T}}

\newcommand{\symplectic}{\mathfrak{S}}

\DeclareMathOperator{\alt}{alt}

\DeclareMathOperator{\Aut}{Aut}
\DeclareMathOperator*{\Average}{Average}

\DeclareMathOperator{\Char}{char}
\DeclareMathOperator{\Cl}{Cl}

\DeclareMathOperator{\coker}{coker}

\DeclareMathOperator{\disc}{disc}

\DeclareMathOperator{\End}{End}

\DeclareMathOperator{\even}{even}

\DeclareMathOperator{\Frac}{Frac}

\DeclareMathOperator{\Gal}{Gal}

\DeclareMathOperator{\height}{ht}
\DeclareMathOperator{\im}{im}

\DeclareMathOperator{\odd}{odd}
\DeclareMathOperator{\ord}{ord}

\DeclareMathOperator{\Prob}{Prob}

\DeclareMathOperator{\rank}{rank}

\DeclareMathOperator{\rk}{rk}

\DeclareMathOperator{\Sel}{Sel}

\newcommand{\tors}{{\operatorname{tors}}}

\newcommand{\GL}{\operatorname{GL}}

\newcommand{\M}{\operatorname{M}}

\newcommand{\SO}{\operatorname{SO}}

\newcommand{\injects}{\hookrightarrow}
\newcommand{\intersect}{\cap} % binary intersection
\newcommand{\isom}{\simeq}

\newcommand{\tensor}{\otimes} % binary tensor product
\newcommand{\Union}{\bigcup} % union of a collection

\newcommand{\squarepairing}{[\;\,{,}\,\;]}
\newcommand{\anglepairing}{\langle \;\,{,}\,\; \rangle}
\newcommand{\parenthesispairing}{(\;\,{,}\,\;)}
\newcommand{\ltwonorm}{|\;\,|}

\newcommand{\sumprime}{\sideset{}{^{'}}{\sum}}
\newcommand{\psmod}[1]{~(\textup{\text{mod}}~{#1})}

\newcommand{\bL}{\Lambda}
\newcommand{\ra}{\rightarrow}

\newcommand{\ProbE}{\mu}
\newcommand{\ProbK}{\mu}

\newtheorem{theorem}[equation]{Theorem}
\newtheorem{lemma}[equation]{Lemma}
\newtheorem{corollary}[equation]{Corollary}
\newtheorem{proposition}[equation]{Proposition}

\theoremstyle{definition}
\newtheorem{definition}[equation]{Definition}
\newtheorem{question}[equation]{Question}
\newtheorem{conjecture}[equation]{Conjecture}

\theoremstyle{remark}
\newtheorem{remark}[equation]{Remark}

\usepackage{microtype}

\usepackage[
	backref,
	pdfauthor={Jennifer Park, Bjorn Poonen, John Voight, Melanie Matchett Wood},
	pdftitle={A heuristic for boundedness of ranks of elliptic curves},
]{hyperref}

\numberwithin{equation}{subsection}

\usepackage[alphabetic,backrefs,lite,nobysame]{amsrefs} % for bibliography

\begin{document}

\title{A heuristic for boundedness of ranks of elliptic~curves}
\subjclass[2010]{Primary 11G05; Secondary 11G40, 11P21, 14G25}
\keywords{Elliptic curve, rank, Shafarevich--Tate group}

\author{Jennifer Park}
\address{Department of Mathematics, University of Michigan, Ann Arbor, MI, USA}
\email{jmypark@umich.edu}
\urladdr{\url{www-personal.umich.edu/~jmypark/}}

\author{Bjorn Poonen}
\address{Department of Mathematics, Massachusetts Institute of Technology, Cambridge, MA 02139-4307, USA}
\email{poonen@math.mit.edu}
\urladdr{\url{http://math.mit.edu/~poonen/}}

\author{John Voight}
\address{Department of Mathematics, Dartmouth College, 6188 Kemeny Hall, Hanover, NH 03755, USA}
\email{jvoight@gmail.com}
\urladdr{\url{http://www.math.dartmouth.edu/~jvoight/}}

\author{Melanie Matchett Wood}
\address{Department of Mathematics,
University of Wisconsin-Madison, 480 Lincoln Drive,
Madison, WI 53705, USA} 
\email{mmwood@math.wisc.edu}
\urladdr{\url{http://www.math.wisc.edu/~mmwood/}}

\date{July 10, 2018}

\begin{abstract}
We present a heuristic that suggests 
that ranks of elliptic curves $E$ over $\Q$ are bounded.
In fact, it suggests that there are only finitely many $E$ 
of rank greater than~$21$.
Our heuristic is based on modeling the ranks and Shafarevich--Tate groups
of elliptic curves simultaneously,
and relies on a theorem counting alternating integer matrices
of specified rank.
We also discuss analogues for elliptic curves over other global fields.
\end{abstract}

\maketitle

\tableofcontents

%****************************************************************************
\section{Introduction}\label{S:introduction}

\subsection{A new model}
\label{S:a new model}

The set $E(\Q)$ of rational points of an elliptic curve $E$ over $\Q$ 
has the structure of an abelian group.
Mordell~\cite{Mordell1922} proved in 1922 that $E(\Q)$ is finitely generated,
so its rank $\rk E(\Q)$ is finite.
Even before this, in 1901, 
Poincar\'e~\cite{Poincare1901}*{p.~173} essentially asked 
for the possibilities for $\rk E(\Q)$ as $E$ varies.  
Implicit in this is the question of boundedness:
Does there exist $B \in \Z_{\ge 0}$ 
such that for every elliptic curve $E$ over $\Q$ one has $\rk E(\Q) \leq B$?

In this article, we present a probabilistic model providing a heuristic
for the arithmetic of elliptic curves, and we prove theorems about 
the model that suggest that $\rk E(\Q) \le 21$ 
for all but finitely many elliptic curves $E$.

Our model is inspired in part by the Cohen--Lenstra heuristics
for class groups~\cite{Cohen-Lenstra1984}, 
as reinterpreted by Friedman and Washington~\cite{Friedman-Washington1989}.
These heuristics predict that for a fixed odd prime $p$, 
the distribution of the $p$-primary part of the class group of 
a varying imaginary quadratic field is 
equal to the limit as $n \to \infty$ of the distribution of
the cokernel of the homomorphism $\Z_p^n \stackrel{A}\to \Z_p^n$
given by a random matrix $A \in \M_n(\Z_p)$;
see Section~\ref{S:Cohen-Lenstra} for the precise conjecture.
In analogy, and in agreement with 
conjectures of Delaunay~\cites{Delaunay2001,Delaunay2007,Delaunay-Jouhet2014a},
Bhargava, Kane, Lenstra, Poonen, and Rains
\cite{Bhargava-Kane-Lenstra-Poonen-Rains2015} 
predicted that for a fixed prime $p$ and $r \in \Z_{\ge 0}$, 
the distribution of 
the $p$-primary part of the Shafarevich--Tate group $\Sha(E)$
as $E$ varies over rank~$r$ elliptic curves over $\Q$ ordered by height
equals the limit as $n \to \infty$ 
(through integers of the same parity as $r$) 
of the distribution of $\coker A$
for a random alternating matrix $A \in \M_n(\Z_p)$
subject to the condition $\rk_{\Z_p}(\ker A) = r$; 
see Section~\ref{S:heuristics for Sha} for the precise conjecture
and the evidence for it.

If imposing the condition $\rk_{\Z_p}(\ker A) = r$ 
yields a distribution conjecturally associated to the curves of rank~$r$,
then naturally we guess that if we choose $A$ at random from
the space $\M_n(\Z_p)_{\alt}$ of all alternating matrices
\emph{without} imposing such a condition,
then the distribution of $\rk_{\Z_p}(\ker A)$ tends as $n \to \infty$
to the distribution of the rank of an elliptic curve.
This cannot be quite right, however: 
since an alternating matrix always has even rank,
the parity of $n$ dictates the parity of $\rk_{\Z_p}(\ker A)$.
But if we choose $n$ uniformly at random from 
$\{\lceil \nnn \rceil,\lceil \nnn \rceil+1\}$ 
(with $\nnn \to \infty$),
then we find that $\rk_{\Z_p}(\ker A)$ 
equals $0$ or $1$ with probability $50\%$ each,
and $\rk_{\Z_p}(\ker A)\ge 2$ with probability~$0\%$;
for example, when $n$ is even, 
we have $\rk_{\Z_p}(\ker A)=0$ unless $\det A = 0$,
and $\det A = 0$ holds only when $A$ lies on a 
(measure~$0$) hypersurface in the space $\M_n(\Z_p)_{\alt}$ 
of all alternating matrices.
This $50\%$--$50\%$--$0\%$ conclusion 
matches the elliptic curve rank behavior conjectured 
for quadratic twist families 
by Goldfeld \cite{Goldfeld1979}*{Conjecture~B} 
and Katz and Sarnak \cites{Katz-Sarnak1999a,Katz-Sarnak1999b}.

So far, however, this model does not predict anything about 
the number of curves of each rank $\ge 2$
except to say that asymptotically they should amount to $0\%$ of curves.
Instead of sampling from $\M_n(\Z_p)$,
we could sample from the set 
$\M_n(\Z)_{\alt,\le X}$ of alternating \emph{integer} matrices
whose entries have absolute values bounded by $X$,
and study 
\[
	\lim_{X \to \infty} 
	\Prob \left( \rk(\ker A) = r \mid A \in \M_n(\Z)_{\alt,\le X} \right),
\]
but this again would be $0$ for each $r \ge 2$.
To obtain finer information, instead of taking the limit as $X \to \infty$,
we let $X$ \emph{depend on the height $H$ of the elliptic curve being modeled};
similarly, we let $\nnn$ grow with $H$ 
so that the random integer $n$ grows too.
Now for each $r \ge 2$,
the event $\rk(\ker A) = r$ occurs with positive probability depending on $H$, 
and we can estimate for how many elliptic curves of height up to $H$ 
the event occurs.

To specify the model completely, we must specify the functions
$\nnn(H)$ and $X(H)$; 
actually, it will turn out that specifying $X(H)^{\nnn(H)}$ is enough
for the conclusions we want to draw.
We calibrate $X(H)^{\nnn(H)}$
so that the resulting prediction for the expected size of $\Sha(E)$
for curves of height up to $H$ agrees with 
theorems and conjectures about this expected size;
this suggests requiring $X(H)^{\nnn(H)} = H^{1/12+o(1)}$.

Our model is summarized as follows.
Fix increasing functions $\nnn(H)$ and $X(H)$ such that 
$X(H)^{\nnn(H)} = H^{1/12+o(1)}$ as $H \to \infty$.
(For technical reasons, we also require $\nnn(H)$ to grow 
sufficiently slowly.)
To model an elliptic curve $E$ of height $H$:
\begin{enumerate}[\upshape 1.]
\item 
Choose $n$ uniformly at random from the pair 
$\{\lceil \nnn(H) \rceil,\lceil \nnn(H) \rceil+1\}$.
\item 
Choose $A_E \in \M_n(\Z)_{\alt}$ 
with entries bounded by $X(H)$ in absolute value, uniformly at random.
\end{enumerate}
Then $(\coker A_E)_{\tors}$ models $\Sha(E)$, and $\rk(\ker A_E)$ models $\rk E(\Q)$.

Thus, heuristically, for an elliptic curve $E$ of height $H$,
the ``probability'' that $\rk E(\Q) \ge r$ 
should be $\Prob(\rk(\ker A_E) \ge r)$.
We prove that for any fixed $r \ge 1$,
the latter probability is $H^{-(r-1)/24 + o(1)}$ as $H \to \infty$ 
(Theorem~\ref{thm:EskinKatznelsonAlternating}).
In other words, for each increase in rank beyond $1$, 
the probability of attaining that rank drops by a factor of about $H^{1/24}$.
Summing the probabilities $H^{-(r-1)/24 + o(1)}$ 
over all elliptic curves $E$ over $\Q$
yields a prediction for the expected number of curves of rank $\ge r$.
It turns out that the sum diverges for $r<21$ and converges for $r>21$.
The latter suggests that there are only finitely many $E$ over $\Q$
with $\rk E(\Q) > 21$.%
\footnote{On the other hand, Elkies~\cite{Elkies2006} proved that
	there exist infinitely many $E$ of rank at least~$19$.}
Summing instead over elliptic curves of height up to $H$
leads to the prediction that for $1 \le r \le 20$,
the number of $E$ of height up to $H$ satisfying $\rk E(\Q) \ge r$
is $H^{(21-r)/24+o(1)}$ as $H \to \infty$.

In order to separate as much as possible what is proved 
from what is conjectured,
we express the model in terms of random variables serving as proxies
for the rank and $\Sha$ of each elliptic curve,
and prove unconditional theorems about these random variables 
before conjecturing that the conclusions
of these theorems are valid also for the \emph{actual} ranks and $\Sha$.
(This methodology is analogous to that of the Cram\'{e}r model, 
which models the set of prime numbers by a random set $P$ 
that includes each $n>2$ independently with probability $1/\log n$; 
see, e.g., the exposition by Granville \cite{Granville1995}.)

For example, we prove the following unconditional result 
(Theorem~\ref{T:rank 21}).

\begin{theorem}
For each elliptic curve $E$ over $\Q$, 
independently choose a random matrix $A_E$ 
according to the model defined above,
and let $\rk'_E$ denote the random variable $\rk(\ker A_E)$.
Then the following hold with probability~$1$:
\begin{enumerate}[\upshape (a)]
\item
All but finitely many $E$ satisfy $\rk'_E \le 21$.
\item
For $1 \le r \le 20$, we have 
$\#\{ E : \height E \le H \textup{ and } \rk'_E \ge r \} = H^{(21-r)/24+o(1)}$.
\item 
We have $\#\{ E : \height E \le H \textup{ and } \rk'_E \ge 21 \} \le H^{o(1)}$.
\end{enumerate}
\end{theorem}

\begin{remark}
Our heuristic explains what should be expected
if there are no significant phenomena in the arithmetic of elliptic curves
beyond those incorporated in the model.
It still could be, however, that there are special families 
of elliptic curves that behave differently for arithmetic reasons,
just as there can be special subvarieties
in the Batyrev--Manin conjectures on the number of rational points
of bounded height on varieties~\cite{Batyrev-Manin1990}.  
When we generalize to global fields in Section~\ref{S:global fields},
we \emph{will} need to exclude some families of curves.
\end{remark}

\begin{remark}
In fact, the known constructions of elliptic curves over $\Q$ of high rank
proceed by starting with a parametric family with high rank generically,
and then finding specializations of even higher rank.
As Elkies points out, 
one cannot say that our heuristic for boundedness (let alone $21$) 
is convincing until one refines the model to predict the rank distribution in
such parametric families.
One plausible heuristic is that 
for a family with generic rank~$r_0$ and varying root number, 
the probability that a curve of height about $H$ in the family
has rank $r_0+s$
is comparable (up to a factor $H^{o(1)}$) 
to the probability that an arbitrary curve of height about~$H$
has rank~$s$.
Although we cannot justify this directly, we can argue by analogy:
the distribution of $p$-Selmer rank in 
certain families with generic rank~$r_0$
is conjecturally obtained simply by shifting
the Selmer rank distribution for all elliptic curves
by $r_0$ \cite{Poonen-Rains2012-selmer}*{Remark~4.17}.
\end{remark}

\begin{remark}
Venkatesh and Ellenberg~\cite{Venkatesh-Ellenberg2010}*{Section~4.1} 
observed that from the arithmetic of an imaginary quadratic field
one can naturally construct an integer square matrix whose cokernel is
the class group; see Section~\ref{S:approximating a class group}.
In contrast, 
we do not know of any structure in the arithmetic of elliptic curves
that suggests the model above for $\rk E(\Q)$;
in particular, we do not yet see a natural alternating matrix 
in the arithmetic of an elliptic curve.
Our reason for using an alternating matrix is 
instead in the spirit of Occam's razor:
the model of alternating matrices over $\Z_p$ 
proposed in \cite{Bhargava-Kane-Lenstra-Poonen-Rains2015} 
is the simplest model we know of that models simultaneously 
the rank of an elliptic curve $E$ and $\Sha$ 
(more precisely, $\Sha(E)[p^{\infty}]$). 
\end{remark}

\begin{remark}
\label{R:Deninger}
Deninger too has conjectured that $\rk E(\Q)$
is naturally the dimension of the kernel of an alternating linear 
map~\cite{Deninger2010}*{Example~5}.
Specifically, in an attempt to explain the Riemann hypothesis for $L(E,s)$,
he conjectured the existence of an infinite-dimensional $\R$-vector space $H_E$
and an endomorphism $\theta \in \End H_E$ such that
\begin{itemize}
\item for any $\rho \in \C$,
the endomorphism $\theta - \rho \in \End(H_E \tensor_\R \C)$ satisfies
$\dim_\C \ker(\theta-\rho) = \ord_{s=\rho} L(E,s)$, and
\item the endomorphism $\theta-1$ is alternating 
with respect to an inner product on $H_E$.
\end{itemize}
If these exist and the Birch and Swinnerton-Dyer conjecture 
is true, then $\rk E(\Q) = \dim \ker(\theta-1)$.
\end{remark}

\subsection{Outline of the paper}

Section~\ref{S:NandC} introduces some notation that will be used
throughout the rest of the paper.
Section~\ref{S:history} surveys some of 
the history regarding ranks of elliptic curves.
Sections \ref{S:Cohen-Lenstra} and~\ref{S:heuristics for Sha}
discuss heuristics for class groups and Shafarevich--Tate groups,
respectively, in terms of cokernels of matrices; 
the former heuristics are not logically necessary for our arguments,
but they serve as the basis for an analogy.
In Section~\ref{S:average Sha} 
we prove theorems to help us predict the average size of $\Sha$;
the idea, due to Lang~\cite{Lang1983-conjectured}, 
is to solve for this size in the Birch and Swinnerton-Dyer conjecture.
These theorems will guide the setting of parameters in our model.
Section~\ref{S:heuristics for ranks} presents the model itself,
and proves unconditional theorems about the random variables in it,
while Section~\ref{S:consequences} conjectures that the conclusions
of these theorems are valid also for the actual ranks and $\Sha$.
One of the statements in Section~\ref{S:heuristics for ranks} 
depends on Theorem~\ref{thm:EskinKatznelsonAlternating},
whose proof is postponed to Section~\ref{S:counting matrices}
so as not to interrupt the flow leading to the main 
conclusions and conjectures in Sections \ref{S:heuristics for ranks}
and~\ref{S:consequences}.  
Section~\ref{S:computational evidence} presents some computational
evidence for our heuristic.
Section~\ref{S:further} discusses some further questions.
Finally, in Section~\ref{S:global fields} 
we discuss analogues of our heuristic for number fields $K$ larger than $\Q$ 
and for global function fields such as $\F_p(t)$.
In particular, we investigate whether our heuristic
predicts a value for $B_K \colonequals \limsup_{E/K} \rk E(K)$.
Also, using either 
Heegner points in anticyclotomic extensions of imaginary quadratic fields,
or recent work of Bhargava, Skinner, 
and Zhang~\cite{Bhargava-Skinner-Zhang-BSD-preprint}
combined with the multidimensional density Hales--Jewett theorem,
one can prove the existence of number fields $K$
for which $B_K$ grows at least linearly in $[K:\Q]$.

%****************************************************************************
\subsection*{Acknowledgements} 

We thank 
Brian Conrey, 
Henri Darmon, 
Noam Elkies, 
Daniel Erman, 
Derek Garton, 
Marc Hindry, 
Barry Mazur, 
Mark Watkins,
and 
Peter Winkler 
for discussions.  
We thank the referee for suggestions to improve 
an earlier draft of this manuscript.
This work began at the workshop 
\emph{Arithmetic statistics over finite fields and function fields} 
at the American Institute of Mathematics, 
and the authors thank AIM for its support.

The first author was supported in part by 
a Natural Sciences and Engineering Research Council Postdoctoral Fellowship.   
The second author was supported in part by 
National Science Foundation (NSF) grants DMS-1069236 and DMS-1601946 
and Simons Foundation grants \#340694, \#402472, and \#550033.
The third author was supported by an NSF CAREER Award (DMS-1151047)
and a Simons Collaboration Grant (550029).
The fourth author was supported by 
an American Institute of Mathematics Five-Year Fellowship, 
NSF grants DMS-1147782, DMS-1301690, and CAREER Award DMS-1652116, 
a Packard Fellowship for Science and Engineering,
 a Sloan Research Fellowship, and a Vilas Early Career Investigator Award.

%****************************************************************************
\section{Notation and conventions}\label{S:NandC}

We make many estimates of functions of several variables.
If ${x}=(x_1,\ldots,x_m)$ and ${a}=(a_1,\ldots,a_n)$,
we write $f({x},{a}) \ll_{{a}} g({x},{a})$
to mean that there exists a positive-valued function $C({a})$ 
such that $f({x},{a}) \le C({a}) g({x},{a})$ 
for all values of $({x},{a})$ we consider.  
We write $f({x},{a}) \asymp_{{a}} g({x},{a})$ to mean
$f({x},{a}) \ll_{{a}} g({x},{a})$ and
$g({x},{a}) \ll_{{a}} f({x},{a})$.
When ``$o(1)$'' appears in a sentence with a variable $H$ going to infinity, 
our interpretation is that there exists a function $f(H)$, 
tending to $0$ as $H \to \infty$,
such that replacing the $o(1)$ by $f(H)$ makes the entire sentence true.

Let $G$ be an abelian group.  
For $n \in \Z_{\ge 1}$, 
let $G[n] \colonequals \{x \in G: nx=0\}$.
For $p$ prime, define 
$G[p^\infty] \colonequals \Union_{m \ge 1} G[p^m]$,
and define the \defi{$p$-rank} of $G$ to be $\dim_{\F_p} G[p]$.

Let $R$ be a commutative ring.  
For $n \in \Z_{\ge 0}$, 
let $\M_n(R)$ be 
the set of $n \times n$ matrices with entries in $R$.
For $X \in \R_{>0}$, let 
$\M_n(\Z)_{\le X} \subseteq \M_n(\Z)$
be the subset of 
matrices whose entries have absolute value less than or equal to $X$.
A matrix $A \in \M_n(R)$ is \defi{alternating} 
if $A^T=-A$ and all the diagonal entries are $0$ 
(if $2$ is not a zero divisor in $R$, 
then the skew-symmetry condition $A^T=-A$ suffices).
Let $\M_n(R)_{\alt}$ be the set of alternating matrices, and let 
$\M_n(\Z)_{\alt,\leq X} \colonequals \M_n(\Z)_{\alt} \cap \M_n(\Z)_{\leq X}$.

For a subset $S \subseteq \M_n(\Z_p)$, 
define $\Prob(S) = \Prob(S \mid A \in \M_n(\Z_p))$
as the probability of $S$ with respect to the normalized Haar measure 
on the compact group $\M_n(\Z_p)$.

Let $R$ be an integral domain, 
and let $K \colonequals \Frac R$ be the field of fractions.
If $M$ is a finitely generated module over $R$,
define $\rk M \colonequals \dim_K (M \tensor_R K)$.
For $A \in \M_n(R)$, let $\rank A$ denote the rank of the matrix,
so $\rank A = n - \rk(\ker A)$.

Finally, because both proven statements and conjectured statements 
play an important role in this paper, in order to distinguish the two,
any unproven or conjectural (in)equality in a displayed equation comes with a question 
mark over the symbol, as in $\stackrel{?}=$.

%****************************************************************************
\section{History}
\label{S:history}

\subsection{Brief history of boundedness guesses}

Many authors have proposed guesses as to whether ranks of elliptic
curves over $\Q$ are bounded, and the consensus seems to
have shifted over time.

Early researchers guessed that ranks were bounded.
In 1950, N\'eron wrote 
``L'existence de cette borne est \dots consid\'er\'ee comme probable'' 
\cite{Poincare1950-Oeuvres5}*{p.~495, end of footnote~(${}^3$)}, 
even though he himself proved the existence of elliptic curves 
of rank $\ge 11$ \cite{Neron1956}.
Honda conjectured in 1960
that for any abelian variety $A$ over $\Q$,
there is a constant $c_A$ such that $\rk A(K) \le c_A [K:\Q]$
for every number field $K$ \cite{Honda1960}*{p.~98}%
\footnote{Honda wrote $=$ instead of $\le$, but almost certainly $\le$ was intended.};
this would imply that ranks are bounded in the family of quadratic
twists of any elliptic curve over $\Q$.

But from the mid-1960s to the present, it seems that most experts
conjectured unboundedness.
Cassels in a 1966 survey article \cite{Cassels1966-diophantine}*{p.~257} 
wrote ``it has been widely conjectured that there is an upper bound
for the rank depending only on the groundfield.  
    % DO NOT REMOVE: The lack of space in groundfield is from the original.
This seems to me implausible
because the theory makes it clear that an abelian variety can only
have high rank if it is defined by equations with very large coefficients.'' 
Tate~\cite{Tate1974}*{p.~194} wrote 
``I would guess that there is no bound on the rank.''
Mestre,
who developed a method for finding elliptic curves of high rank, 
wrote ``Au vu de cette m\'ethode, il semble que l'on puisse s\'erieusement
conjecturer que le rang des courbes elliptiques d\'efinies sur $\Q$ n'est
pas born\'e'' \cite{Mestre1982},
and proved a rank bound depending on the conductor $N$ of $E$,
namely $O(\log N)$ unconditionally \cite{Mestre1986}*{II.1.1},
and $O(\log N/\log \log N)$ conditionally 
on the Riemann hypothesis for $L(E,s)$ \cite{Mestre1986}*{II.1.2}.
Silverman \cite{SilvermanAEC2009}*{Conjecture~10.1}
wrote that it is a ``folklore conjecture'' that ranks are unbounded.
In 1992 Brumer~\cite{Brumer1992}*{Section~1} 
wrote ``Today, it is believed that the rank is unbounded,''
and noted that the available numerical data was
``not incompatible with the possibility
that, for each $r$, some positive proportion of all curves might have
rank at least $r$.''  

Here are two possible reasons for this opinion shift towards unboundedness:
\begin{enumerate}[\upshape 1.]
\item Tate and Shafarevich \cite{Tate-Shafarevich1967}
and Ulmer \cite{Ulmer2002} constructed families of elliptic curves
over $\F_p(t)$ in which the rank is unbounded.
\item Every few years, the proved lower bound on the maximum rank
of an elliptic curve over $\Q$ increased:
see \cite{Rubin-Silverberg2002}*{Section~3} for the history up to 2002.
The current record is held by Elkies \cite{Elkies2006},
who found an elliptic curve $E$ over $\Q$ of rank $\ge 28$,
and an infinite family of elliptic curves over $\Q$ of rank $\ge 19$.
\end{enumerate}

Some authors have even proposed a rate at which rank can grow 
relative to the conductor~$N$:
\begin{itemize}
\item Ulmer's examples over $\F_p(t)$ attained 
Brumer's (unconditional) function field analogue 
\cite{Brumer1992}*{Proposition~6.9} 
of Mestre's conditional upper bound $O(\log N/\log \log N)$.
This led Ulmer \cite{Ulmer2002}*{Conjecture~10.5} to conjecture 
that Mestre's conditional bound would be attained over $\Q$,
that is, that 
\[
	\limsup_{N \to \infty} \frac{\rk E(\Q)}{\log N / \log \log N} \stackrel{?}> 0.
\]
\item 
On the other hand, 
Farmer, Gonek, and Hughes \cite{Farmer-Gonek-Hughes2007}*{(5.20)}, 
based on conjectures for the maximal size of critical values  
and the error term in the number of zeros up to a given height
for families of $L$-functions, suggest that
\begin{equation}
\limsup_{N \to \infty} \frac{ \rk E(\Q)}{\sqrt{\log N\log\log N}} \stackrel{?}= 1,
\end{equation}
in contradiction to Ulmer's conjecture.
\end{itemize}

\subsection{Previous heuristics for boundedness}
\label{S:previous heuristics}

\begin{enumerate}[\upshape (a)]
\item 
Rubin and Silverberg~\cite{Rubin-Silverberg2000}*{Remarks 5.1 and~5.2} 
gave a heuristic based on the expected size of squarefree parts
of binary quartic forms.
They showed that if certain lattices they define were randomly distributed,
then ranks in a family of quadratic twists of a fixed elliptic curve $E$
would be bounded by $8$.
As they knew, however, the conclusion is wrong for some curves $E$, 
e.g., any $E$ of rank greater than $8$.
Presumably this explains why they did not conjecture
boundedness of rank based on this heuristic.
\item 
Granville gave a heuristic, 
discussed in~\cite{Watkins-et-al2014}*{Section~11} 
and further developed in~\cite{Watkins-discursus}, 
based on estimating the number of integer solutions of bounded height
to the equation defining a family of elliptic curves.
His observation was that a single elliptic curve of high rank
would by itself contribute more integer solutions than should be 
expected for the whole family.
Watkins \cite{Watkins-et-al2014}*{Section~11.4} 
writes that similar ideas would lead to 
the conclusion that all but finitely many
elliptic curves $E$ satisfy $\rk E(\Q) \le 21$.
See also comments by Conrey, Rubinstein, Snaith, and Watkins
\cite{Conrey-Rubinstein-Snaith-Watkins2007}*{Section~1.3}.
\end{enumerate}
These two approaches seem completely unrelated to ours.

\subsection{Conjectures for rank~2 asymptotics}
\label{S:density of rank 2}

For each elliptic curve $E$ over $\Q$, 
let $L(E,s)$ be the $L$-function of $E$, 
and let $w(E) \in \{1,-1\}$ be the sign of its functional equation, 
or equivalently, the global root number.
The Birch and Swinnerton-Dyer conjecture would imply
the parity conjecture, that $w(E) \stackrel{?}= (-1)^{\rk E(\Q)}$.

Much of the literature on the distribution of ranks of elliptic curves
focuses on quadratic twist families.
Fix an elliptic curve $E$ over $\Q$.
Let $d$ range over fundamental discriminants in $\Z$.
For each $d$, let $E_d$ denote the twist of $E_1$ by $\Q(\sqrt{d})/\Q$.
Given $r \in \Z_{\ge 0}$ and $D>0$, define
\begin{align*}
	N_{\ge r}(D) &\colonequals 
	\#\{d : |d| \le D, \; \rk E_d(\Q) \ge r \} \\
	N_{\ge r, \even}(D) &\colonequals 
	\#\{d : |d| \le D, \; \rk E_d(\Q) \ge r, \textup{ and } w(E_d)=+1 \} \\
	N_{\ge r, \odd}(D) &\colonequals 
	\#\{d : |d| \le D, \; \rk E_d(\Q) \ge r, \textup{ and } w(E_d)=-1 \}.
\end{align*}

There are many different approaches for estimating $N_{\ge 2,\even}(D)$,
listed below, but they all lead to the conjecture that
\begin{equation}\label{E:rk 2 conj}
	N_{\ge 2,\even}(D) \stackrel{?}{=} D^{3/4+o(1)}.
\end{equation}  
In other words, the prediction is that for $d$ such that $w(E_d)=+1$,
the probability that $\rk E(\Q) \ge 2$ should be about $d^{-1/4}$.
Since $\height E_d \asymp d^6$, this prediction corresponds 
to a probability of $H^{-1/24}$ for an elliptic curve of height~$H$.

\begin{enumerate}[\upshape (a)]
\item 
Let $E$ be an elliptic curve over $\Q$ with $w(E)=+1$.  
Then Waldspurger's work \cite{Waldspurger1981}*{Corollaire~2, p.~379} 
combined with the modularity of $E$,
yields a weight~$3/2$ cusp form $f = \sum a_n q^n$ 
such that for all odd fundamental discriminants 
$d<0$ coprime to the conductor of $E$,
we have $a_{|d|}=0$ if and only if $L(E_d,1)=0$ 
(see also Ono and Skinner 
\cite{Ono-Skinner1998-Inventiones}*{Section 2, Proof of (2a,b)} 
and Gross \cite{Gross1987}*{Proposition~13.5}).
When $w(E_d)=+1$, the condition $L(E_d,1)=0$
is equivalent to $\ord_{s=1} L(E_d,s) \ge 2$,
which is equivalent to $\rk E_d(\Q) \ge 2$
if the Birch and Swinnerton-Dyer conjecture holds.
The Ramanujan conjecture \cite{Sarnak1990}*{Conjecture~1.3.4} predicts that 
$a_{|d|}$ is an integer satisfying $|a_{|d|}| \le |d|^{1/4+o(1)}$, 
so one might expect that $a_{|d|}=0$ occurs 
with ``probability'' $|d|^{-1/4+o(1)}$.
If we ignore the conditions on the sign, parity, and coprimality of $d$, 
then summing over $|d| \le D$ suggests the guess 
$N_{\ge 2,\even}(D) \stackrel{?}= D^{3/4+o(1)}$.
This heuristic argument has been attributed to Sarnak 
\cite{Conrey-Keating-Rubinstein-Snaith2002}*{p.~302}.
\item
Conrey, Keating, Rubinstein, and Snaith 
\cite{Conrey-Keating-Rubinstein-Snaith2002} 
used random matrix theory to obtain a conjecture 
more precise than~\eqref{E:rk 2 conj},
namely that there exist $c_E,e_E \in \R$
such that 
\[
	N_{\ge 2,\even}(D) \stackrel{?}= (c_E + o(1)) D^{3/4} (\log D)^{e_E};
\]
later, Delaunay and Watkins \cite{Delaunay-Watkins2007}
explained how to predict $e_E$ in terms of the $2$-torsion of $E$.
The starting point is 
the Katz--Sarnak philosophy \cites{Katz-Sarnak1999a,Katz-Sarnak1999b},
based on a function field analogy,
that $L(E_d,s)$ should be modeled by the characteristic polynomial
of a random matrix from $\SO_{2N}(\R)$ for large $N$
(these random matrices seem unrelated to the 
$p$-adic and integral matrices in our heuristics).
Moment calculations of Keating and Snaith
determined the distribution of the values at $1$ of
the characteristic polynomials \cite{Keating-Snaith2000}*{Section~3.2}.
Conrey, Keating, Rubinstein, and Snaith 
obtained their conjecture by combining this
with a discretization heuristic
(interpreting sufficiently small $L$-values as $0$).

Watkins \cite{Watkins2008-ExpMath} developed a variant
for the family of \emph{all} elliptic curves over $\Q$:
he conjectured that there exists $c>0$ such that 
\[
	\#\{E: \height E \le H, \; w(E)=+1, \textup{ and } \rk E(\Q) \ge 2\} 
	\stackrel{?}= (c+o(1)) H^{19/24}(\log H)^{3/8},
\]
which is a refined version of what our heuristic predicts.
(Watkins counts by discriminant instead of height,
but one of his assumptions is that 
the two counts are comparable \cite{Watkins2008-ExpMath}*{Section~3.4}.)
\item 
Watkins \cite{Watkins2008-ExpMath}*{Section~4.5} 
also gave another argument for $H^{19/24+o(1)}$:
since the number $\Sha_0(E)$ defined in Section~\ref{S:average size of Sha}
is expected to be a square integer of size at most $H^{1/12+o(1)}$ 
(see Theorem~\ref{T:bound on Sha_0}\eqref{I:Sha_0}), 
one can guess it is $0$ about $H^{-1/24+o(1)}$ of the time, 
and there are $\asymp H^{20/24}$ elliptic curves in total.
\item 
Granville's heuristic (see Section~\ref{S:previous heuristics})
would again suggest $H^{19/24+o(1)}$,
according to Watkins \cite{Watkins-discursus}*{Section~6} 
(see also \cite{Watkins2008-ExpMath}*{Section~4.5}). 
\end{enumerate}
Our model introduced in Section~\ref{S:a new model},
based on yet another approach, 
again predicts~\eqref{E:rk 2 conj}.

\subsection{Conjectures for rank~3 asymptotics}
\label{S:density of rank 3}

While the conjectures for $N_{\ge 2,\even}(D)$ are in agreement,
the conjectures in the literature for $N_{\ge 3,\odd}(D)$ are not.
\begin{enumerate}[\upshape (a)]
\item 
Rubin and Silverberg \cite{Rubin-Silverberg2001}*{Theorem~5.4}, 
building on work of Stewart and Top~\cite{Stewart-Top1995}, 
showed that the parity conjecture implies
the lower bound $N_{\ge 3,\odd}(D) \gg D^{1/3}$
for many $E$.
\item
Conrey, Rubinstein, Snaith, and Watkins 
\cite{Conrey-Rubinstein-Snaith-Watkins2007}
used random matrix theory as in 
\cite{Conrey-Keating-Rubinstein-Snaith2002},
but the discretization depends on 
a lower bound $L'(E_d,1) \gg d^{-\theta}$
for analytic rank~$1$ twists $E_d$,
and it is not clear what the best $\theta$ is.
In fact, they proposed three approaches to suggest a value for $\theta$:
\begin{enumerate}[\upshape (1)]
\item The Birch and Swinnerton-Dyer conjecture implies a lower bound
with $\theta=1/2$, which leads to $N_{\ge 3,\odd}(D)$ being only about $D^{1/4}$,
contradicting the conditional theorem of Rubin and Silverberg above 
\cite{Conrey-Rubinstein-Snaith-Watkins2007}*{p.~3}.
\item An analogy with the class number problem suggests 
that the lower bound is valid for any $\theta>0$ 
\cite{Conrey-Rubinstein-Snaith-Watkins2007}*{p.~2};
this leads to $N_{\ge 3,\odd}(D) = D^{1-o(1)}$,
more than what is conjectured for $N_{\ge 2,\even}(D)$!
\item A model involving Heegner points 
(attributed ``largely to Birch'' 
\cite{Conrey-Rubinstein-Snaith-Watkins2007}*{Section~1.2})
again suggests that any $\theta>0$ is valid,
and hence again that $N_{\ge 3,\odd}(D) = D^{1-o(1)}$.
\end{enumerate}
\item 
Conrey, Rubinstein, Snaith, and Watkins
suggest another heuristic at the beginning of 
\cite{Conrey-Rubinstein-Snaith-Watkins2007}*{Section~1.3},
namely that the connection between rank $1$ and rank $3$ twists 
should be the same as between rank $0$ and rank $2$, 
at least to first approximation;
this suggests $N_{\ge 3,\odd}(D) = D^{3/4+o(1)}$.
\item
Granville's heuristic, discussed at the end of 
\cite{Conrey-Rubinstein-Snaith-Watkins2007}*{Section~1.3},
suggests that $N_{\ge 3}(D) \ll D^{2/3+o(1)}$.
\item 
Delaunay and Roblot give heuristics on the moments of regulators 
that suggest $N_{\ge 3,\odd}(D) = D^{1-o(1)}$ 
\cite{Delaunay-Roblot2008}*{p.~608}.  
(See also \cite{Delaunay2005} for related conjectures on the regulators.)
\end{enumerate}

There is also numerical data 
\cites{Elkies2002,Delaunay-Duquesne2003,Watkins2008-JTNB}.
According to Rubin and Silverberg \cite{Rubin-Silverberg2002}*{p.~466},
the numerical data of Elkies suggests that $N_{\ge 3,\odd}(D)$
is about $D^{3/4}$.
Watkins writes in \cite{Watkins2008-JTNB}*{Section~3.2}, however, 
that fitting more extensive data suggests an exponent for $N_{\ge 3,\odd}(D)$
noticeably smaller than the $3/4$ exponent for $N_{\ge 2,\even}(D)$.

Our model, using a single approach that also reproduces 
the well-known rank $2$ conjecture, 
predicts that $N_{\ge 3}(D)=D^{1/2+o(1)}$ and $N_{\ge 3,\odd}(D)=D^{1/2+o(1)}$.  
This prediction is different from all those above,
but it is consistent with the conditional lower bound of Rubin and Silverberg 
and with the heuristic upper bound of Granville.

%****************************************************************************
\section{Cohen--Lenstra heuristics for class groups}
\label{S:Cohen-Lenstra}

In this section, we give a brief exposition of heuristics for class groups,
to motivate Section~\ref{S:heuristics for Sha} by analogy.
The conjectures are due originally to 
Cohen and Lenstra~\cite{Cohen-Lenstra1984} 
(with extensions by Cohen and Martinet~\cite{Cohen-Martinet1990}).
Following Friedman and Washington~\cite{Friedman-Washington1989} 
and Venkatesh and Ellenberg~\cite{Venkatesh-Ellenberg2010}*{Section~4.1}, 
we reinterpret these conjectures in terms of random integer matrices.

\subsection{Class groups as cokernels of integer matrices} 
\label{S:approximating a class group}

Let $K$ be a number field.
Let $I$ be the group of nonzero fractional ideals of $K$.
Let $P$ be the subgroup of $I$ consisting of principal fractional ideals.
The class group $\Cl K \colonequals I/P$ is a finite abelian group.

Let $\calO_K$ be the ring of integers of $K$.
Let $S_\infty$ be the set of \emph{all} archimedean places of $K$.
The Dirichlet unit theorem states that the unit group $\calO_K^\times$ 
is a finitely generated abelian group of rank $u \colonequals \#S_\infty-1$.

Let $S$ be a finite set of places of $K$ containing $S_\infty$.
Let $n \colonequals \#(S-S_\infty)$.
Let $\calO_{K,S}$ be the ring of $S$-integers of $K$.
By the Dirichlet $S$-unit theorem, $\calO_{K,S}^\times$
is a finitely generated abelian group of rank $\#S-1 = n+u$.

Let $I_S$ be the group of fractional ideals 
generated by the (nonarchimedean) primes in $S$.
Let $P_S$ be the subgroup of $I_S$ consisting of principal fractional ideals,
so we obtain an injective homomorphism $I_S/P_S \injects I/P = \Cl K$.
If $S$ is chosen so that its primes generate the finite group $\Cl K$,
then $I_S/P_S \isom I/P = \Cl K$.

The group $I_S$ is a free abelian group of rank $n$.
Since $P_S$ is the image of the homomorphism $\calO_{K,S}^{\times} \to I_S$,
whose kernel is the torsion subgroup of $\calO_{K,S}^\times$,
the group $P_S$ is a free abelian group of rank $n+u$.
If we choose bases,
then we represent $\Cl K$ as the cokernel of a homomorphism 
$\Z^{n+u} \to \Z^n$.
We write this cokernel as $\coker A$ 
for some $n \times (n+u)$ matrix $A$ over $\Z$.
If we view this same $A$ as a matrix over $\Z_p$,
then $\coker(A \colon \Z_p^{n+u} \to \Z_p^n) = (\Cl K)[p^\infty]$.

\begin{remark}
Friedman and Washington \cite{Friedman-Washington1989}
were the first to model (the Sylow $p$-subgroups of) 
class groups as cokernels of matrices,
but they arrived at such a model via a different argument.
Specifically, they considered the function field analogue,
in which case $(\Cl K)[p^\infty]$ for $p \ne \Char K$
can be understood in terms of the action of Frobenius
on the Tate module $T_p J$ of the Jacobian $J$ of a curve
over a finite field.
It was only later that 
Venkatesh and Ellenberg~\cite{Venkatesh-Ellenberg2010}*{Section~4.1}
noticed the connection with the presentation of the class group given above.  
\end{remark}

\subsection{Heuristics for class groups}
\label{S:heuristics for class groups}

Let $\KK$ be the family of all imaginary quadratic fields
up to isomorphism.
What is the distribution of $\Cl K$ as $K$ varies over $\KK$?
To formulate this question precisely, we order the fields
by their discriminant $D \colonequals \disc K$.
For $X>0$, 
let $\KK_{\le X} \colonequals \{K \in \KK : \lvert\disc K\rvert \le X\}$.
Define the \defi{density} of a subset $S \subset \KK$ by
\[
	\ProbK(S) = \ProbK(S \mid K \in \KK) 
	\colonequals 
	\lim_{X \to \infty} \frac{\#(S \intersect \KK_{\le X})}{\#\KK_{\le X}}
\]
when this limit exists.

Hecke, Deuring, and Heilbronn proved that
$\#\Cl K \to \infty$ as $|D| \to \infty$, 
and soon thereafter Siegel proved $\#\Cl K = |D|^{1/2+o(1)}$;
see the appendix to Serre~\cite{SerreMordellWeil1997} for the history.
Therefore, for any finite abelian group $G$,
the set $\{K \in \KK : \Cl K \isom G\}$ is finite,
so $\ProbK(\Cl K \isom G) = 0$.

To get subsets of positive density,
we instead examine the $p$-Sylow subgroup $(\Cl K)[p^\infty]$ for a fixed prime $p \neq 2$ (The case $p=2$ is different because of genus theory;
Gerth~\cite{Gerth1987} formulated analogous conjectures
by considering $(\Cl K)^2[2^\infty]$ instead,
and Fouvry and Kl\"uners \cite{FouvryKlueners2007} 
proved that $(\Cl K)^2[2]$ is distributed as Gerth conjectured.)
For each finite abelian $p$-group $G$,
the density $\ProbK((\Cl K)[p^\infty]\isom G)$ is conjecturally positive,
and there are two conjectures for its value, as follows.
\begin{enumerate}
\item[(\ref{eqn:conj1cl})]
The density is inversely proportional to $\#\Aut G$:
\begin{equation} \label{eqn:conj1cl} \notag \stepcounter{equation}
	\ProbK((\Cl K)[p^\infty]\isom G) 
	\stackrel{?}= 
	\frac{(\#\Aut G)^{-1}}{\eta(p)},
\end{equation}
where the normalization constant $\eta(p)$ 
needed for a probability distribution is given by Hall~\cite{Hall1938} as
\[ 
	\eta(p) 
	\colonequals \sum_{\substack{\textup{finite abelian} \\ \textup{$p$-groups $G$}}} (\#\Aut G)^{-1} 
	= \prod_{i=1}^{\infty} (1-p^{-i})^{-1}.
\]
\item[(\ref{eqn:conj2cl})]
Inspired by Section~\ref{S:approximating a class group}, 
with unit rank $u=\#S_\infty-1=0$, 
one models $(\Cl K)[p^\infty]$ as $(\coker A)[p^\infty]$ for a ``random'' $n \times n$ matrix $A$ over $\Z$ or $\Z_p$:
\begin{equation} \label{eqn:conj2cl} \notag \stepcounter{equation}
\begin{aligned}
	\ProbK((\Cl K)[p^\infty]\isom G) 
	&\stackrel{?}= 
	\lim_{n \to \infty} 
	\lim_{X \to \infty} 
	\frac{\# \{A \in \M_n(\Z)_{\le X} : (\coker A)[p^\infty] \isom G\}}
		{\# \M_n(\Z)_{\le X}} \\
	&= 
	\lim_{n \to \infty} 
	\Prob(\coker A \isom G \mid A \in \M_n(\Z_p)).
\end{aligned}
\end{equation}
(Recall our conventions in Section~\ref{S:NandC} for these probabilities; 
the equality of the probabilities in the last two expressions 
follows from the asymptotic equidistribution of $\Z$ in $\Z_p$.
The equality of the limits in the last two expressions is very robust; it holds when when we replace $A\in \M_n(\Z)$ by drawing $A$ from much more general distributions of integral matrices \cite{Wood-preprint}.)
\end{enumerate}

Conjecture~\eqref{eqn:conj1cl} is due to 
Cohen and Lenstra~\cite{Cohen-Lenstra1984};
they were motivated by numerical data
and the general principle that an object should be counted with weight 
inversely proportional to the size of its automorphism group.  Conjecture~\eqref{eqn:conj2cl} in the second form
\[
	\ProbK((\Cl K)[p^\infty]\isom G) \stackrel{?}= 
	\lim_{n \to \infty} 
	\Prob(\coker A \isom G \mid A \in \M_n(\Z_p))
\]
is due to Friedman and Washington~\cite{Friedman-Washington1989}.

In fact, Conjectures \ref{eqn:conj1cl} and~\ref{eqn:conj2cl}
are equivalent:

\begin{theorem}[Friedman and Washington \cite{Friedman-Washington1989}]
For every finite abelian $p$-group $G$, 
\[
	\lim_{n \to \infty} 
	\Prob(\coker A \isom G \mid A \in \M_n(\Z_p))
	= \frac{(\#\Aut G)^{-1}}{\eta(p)} = \frac{1}{\#\Aut G} \prod_{i=1}^{\infty} (1-p^{-i}).
\]
\end{theorem}

If instead we consider the family of \emph{real} quadratic fields, then 
the unit rank $u$ is $1$, so Section~\ref{S:approximating a class group}
suggests that $\Cl K$ should be modeled by 
the cokernel of an $n \times (n+1)$ matrix, 
in which case there is a similar story to that above.

%****************************************************************************
\section{Heuristics for Shafarevich--Tate groups}
\label{S:heuristics for Sha}

In this section, we consider heuristics for the Shafarevich--Tate group of an elliptic curve over $\Q$, analogous to the heuristics for class groups in the previous section.

\subsection{Elliptic curves}

An elliptic curve $E$ over $\Q$
is isomorphic to the projective closure 
of a curve $y^2=x^3+Ax+B$ for a unique pair of integers $(A,B)$ 
such that there is no prime $p$ such that $p^4 \mid A$ and $p^6 \mid B$.
Conversely, any such pair $(A,B)$ with $4A^3+27B^2 \ne 0$ 
defines an elliptic curve over $\Q$.
Let $\EE$ be the set of elliptic curves of this form, one in each $\Q$-isomorphism class.

Define the \defi{(naive) height} of $E \in \EE$ 
by 
\[ 
	\height E \colonequals \max(|4A^3|,|27B^2|). 
\]
Let $\EE_{\le H} \colonequals \{E \in \EE: \height E \le H\}$.
An elementary sieve argument \cite{Brumer1992}*{Lemma~4.3} shows that 
\begin{equation} \label{eqn:height}
	\#\EE_{\le H} = (\kappa + o(1)) H^{5/6},
\end{equation}
where $\kappa \colonequals 2^{4/3} 3^{-3/2} \zeta(10)^{-1}$.

For a subset $S\subseteq \EE$, we define densities 
\begin{align*}
  \ProbE(S) &\colonequals \lim_{H\rightarrow\infty}
  		\frac{\#(S\cap \EE_{\le H})}{\# \EE_{\le H}} \\
  \ProbE(S \mid \rk E(\Q)=r ) &\colonequals \lim_{H\rightarrow\infty}
  \frac{\#\{ E\in S\cap \EE_{\le H}\mid \rk E(\Q)=r  \}}
       {\#\{ E\in  \EE_{\le H}\mid \rk E(\Q)=r  \}},
\end{align*}
when the limits exist.  

\begin{remark}
If for some $r$, there are no $E \in \EE$ such that $\rk E(\Q) = r$,
then the density $\ProbE(S \mid \rk E(\Q)=r )$ does not exist!
\end{remark}

\begin{remark}
\label{R:counting by discriminant}
Elliptic curves can be ordered in other ways, 
such as by minimal discriminant or conductor.
It is still true that the set of $E \in \EE$
of minimal discriminant or conductor up to $X$ is finite,
but there is no unconditional estimate for its size, 
even though for most $E$ (ordered by height), 
the minimal discriminant and conductor 
are of the same order of magnitude as the height.
See Watkins \cite{Watkins2008-ExpMath}*{Section~4} for further discussion.
Hortsch~\cite{Hortsch-preprint} recently succeeded in counting
elliptic curves of bounded \emph{Faltings height}, however.
\end{remark}

Associated to an elliptic curve $E \in \EE$ are other invariants:
the \defi{$n$-Selmer group} $\Sel_n E$ for each $n \ge 1$
and the \defi{Shafarevich--Tate group} $\Sha(E)$;
see Silverman~\cite{SilvermanAEC1992}*{Chapter~10}.
These invariants are related by an exact sequence
\[
	0 \to \frac{E(\Q)}{nE(\Q)} \to \Sel_n E \to \Sha(E)[n] \to 0
\]
for each $n \ge 1$.
Taking the direct limit as $n$ ranges over powers of a prime $p$ yields the exact sequence
\begin{equation}
\label{E:Selmer-Sha}
	0 \to E(\Q) \tensor \frac{\Q_p}{\Z_p} 
	\to \Sel_{p^{\infty}}E \to \Sha(E)[p^{\infty}] \to 0.
\end{equation}

Instead of trying to predict a distribution for $\rk E(\Q)$ in isolation,
we model all three invariants at once. 
 This lets us check our model against other 
theorems and conjectures in the literature.

\subsection{Symplectic finite abelian groups}
\label{S:symplectic abelian groups}

We will soon focus on $\Sha(E)$, which is an abelian group
with extra structure that we now describe.

\begin{definition}
A \defi{symplectic finite abelian group} is a pair $(G,\squarepairing)$, 
where $G$ is a finite abelian group
and $\squarepairing \colon G \times G \to \Q/\Z$ 
is a nondegenerate alternating pairing.  
\end{definition}

An isomorphism of symplectic finite abelian groups 
is an isomorphism of groups that respects the pairings.
It turns out that if two symplectic finite abelian groups
are isomorphic as abstract groups, 
there is automatically an isomorphism that respects the pairings.
Let $\symplectic$ be a set of symplectic finite abelian groups
containing exactly one from each isomorphism class.
If $J$ is a finite abelian group, 
then $J \times J^\vee$ equipped with a natural pairing 
is a symplectic finite abelian group,
and every symplectic finite abelian group is isomorphic to one of this form.
In particular, symplectic finite abelian groups have square order.

Let $\symplectic_p$ be the set of $G \in \symplectic$
such that $\#G$ is a power of $p$.

\subsection{Distribution of the Shafarevich--Tate group}
\label{S:distribution of Sha}

It is widely conjectured that $\Sha(E)$ is finite.
Cassels~\cite{Cassels1962-IV} constructed a alternating pairing
\[
	\anglepairing \colon \Sha(E) \times \Sha(E) \to \Q/\Z.
\]
He proved also that if $\Sha(E)$ is finite, 
then $\anglepairing$ is nondegenerate.
In this case, $\Sha(E)$ equipped with $\anglepairing$
is a symplectic finite abelian group, and in particular $\#\Sha(E)$ is a square.
This already shows that the distribution of $\Sha(E)$
will be different from the conjectural distribution of class groups
in Section~\ref{S:heuristics for class groups}.

The distribution of class groups conjecturally 
depended on the unit rank of the number field;
analogously, the distribution of $\Sha(E)$ should depend on the rank of $E$.

\begin{question}
\label{Q:distribution of Sha}
Fix a prime $p$.
Given $r \geq 0$ and $G \in \symplectic_p$, 
what is the density
\[ 
	\ProbE(\Sha(E)[p^\infty] \isom G \mid \rk E(\Q)=r ) ?
\]
\end{question}

There are three conjectural answers to this question:
\begin{itemize}
\item[($\DD_r$)] 
Delaunay~\cites{Delaunay2001,Delaunay2007,Delaunay-Jouhet2014a},
in analogy with the Cohen--Lenstra heuristics for class groups,
made conjectures on the distribution of $\Sha(E)$ 
as $E$ varies over elliptic curves of rank~$r$.
He ordered elliptic curves by conductor;
but if we modify his conjectures to order by height,
they imply that the answer to Question~\ref{Q:distribution of Sha} 
is given by the probability measure $\DD_r=\DD_{r,p}$ on $\symplectic_p$
defined by
\begin{equation}
\label{E:DD_r}
        \Prob_{\DD_r}(G) \colonequals
	\frac{\#G^{1-r}}{\# \Aut G}  \prod_{i \ge r+1}(1-p^{1-2i}),
\end{equation}
where $\Aut G$ denotes the group of automorphisms of $G$ 
\emph{that respect the pairing}.
\item[($\TT_r$)] Work of Poonen and Rains
\cite{Poonen-Rains2012-selmer}
and Bhargava, Kane, Lenstra, Poonen, and Rains 
\cite{Bhargava-Kane-Lenstra-Poonen-Rains2015}
predicted the distribution of 
the isomorphism type of
$\Sel_p E$ and the short exact sequence~\eqref{E:Selmer-Sha},
respectively,
and these were shown to be compatible with some known properties
of the arithmetic of~$E$.
{}From these, one extracts a probability measure
$\TT_r$ on $\symplectic_p$ conjectured to model $\Sha(E)[p^\infty]$.  
\item[($\AA_r$)]
The article~\cite{Bhargava-Kane-Lenstra-Poonen-Rains2015},
in analogy with the Friedman--Washington interpretation
of class group heuristics,
proposed also another probability measure, $\AA_r$, 
inspired by the observation that if $A \in \M_n(\Z_p)_{\alt}$,
then $\coker(A:\Z_p^n \to \Z_p^n)_{\tors}$ 
is naturally a symplectic finite abelian $p$-group.
Specifically, for $n \equiv r \pmod{2}$,
there is a canonical probability measure 
on the set 
\[ 
	\{A \in \M_n(\Z_p)_{\alt} : \rk_{\Z_p}(\ker A) = r \}, 
\]
and $G \in \symplectic_p$, we let $\AA_{n,r}(G)$ be the measure of 
\[ 
	\{A \in \M_n(\Z_p)_{\alt} : \rk_{\Z_p}(\ker A) = r 
	\text{ and } (\coker A)_{\tors} \isom G \}.
\]
Then the formula 
\[
	\AA_r(G) \colonequals 
	\lim_{\substack{n \to \infty \\ n \equiv r \psmod{2}}}
	\AA_{n,r}(G)
\]
defines a probability measure~$\AA_r$ on $\symplectic_p$.
\end{itemize}

\begin{theorem}[\cite{Bhargava-Kane-Lenstra-Poonen-Rains2015}*{Theorems 1.6(c) and~1.10(b)}] \label{T:three same}
The probability measures $\DD_r$, $\TT_r$, $\AA_r$ coincide.
\end{theorem}

\begin{remark}
\label{R:large and small Sha}
Conjecturally, $\Sha(E)$ is large on average when $r=0$
and small when $r \ge 1$,
just as class groups of quadratic fields are large 
if the field is imaginary ($u=0$)
and conjecturally small if the field is real ($u=1$).
More precisely, it follows from Delaunay's conjectures on $\Sha(E)$ 
mentioned above that
$\ProbE(\#\Sha(E) \le B \mid \rk E(\Q) = 0) = 0$
for all $B>0$,
but  for each $r \ge 1$ that
$\ProbE(\#\Sha(E) \le B \mid \rk E(\Q) = r) \to 1$ as $B \to \infty$.
In fact, for fixed $r \ge 1$, 
Delaunay's conjectures predict for each $G \in \symplectic$
that $\ProbE(\Sha(E) \isom G \mid \rk E(\Q)=r)$ 
is an explicit positive number, 
and these numbers define a measure on $\symplectic$
that agrees with the product over all primes of the measures $\DD_r$. 
See also \cite{Bhargava-Kane-Lenstra-Poonen-Rains2015}*{Section~5.6} 
for further discussion.
\end{remark}

%****************************************************************************
\section{Average size of the Shafarevich--Tate group}
\label{S:average Sha}

Section~\ref{S:heuristics for ranks} will propose a model
for ranks and $\Sha$.
To set the parameters in that model, we will need to know
the typical size of $\#\Sha(E)$ for a rank~$0$ elliptic curve 
of height about~$H$.
Our approach to estimating $\#\Sha(E)$ 
is similar to that in Lang~\cite{Lang1983-conjectured};
see also work of 
Goldfeld and Szpiro \cite{Goldfeld-Szpiro1995}, 
de Weger \cite{deWeger1998}, 
Hindry \cite{Hindry2007}, 
Watkins \cite{Watkins2008-ExpMath},
and Hindry and Pacheco \cite{Hindry-Pacheco2016}.
Although more precise results are known, we provide a streamlined 
version of these estimates that is sufficient for our purposes.

\subsection{Size of the real period}
\label{S:real period}

\begin{lemma}\label{L:real period}
Let $A, B \in \R$ satisfy $4A^3+27B^2\ne 0$,
so that the equation $y^2=x^3+Ax+B$ defines an elliptic curve $E$ over $\R$.
Let $\Delta \colonequals -16(4A^3+27B^2)$, 
let $H \colonequals \max(|4A^3|,|27B^2|)$,
and let 
$\Omega \colonequals \int_{E(\R)} \bigl\lvert\frac{dx}{2y} \bigr\rvert$.
Then
\[
	H^{-1/12} \ll \Omega \ll H^{-1/12} \log(H/|\Delta|).
\]
\end{lemma}

\begin{proof}
Changing $(A,B)$ to $(\lambda^4 A,\lambda^6 B)$ 
with $\lambda \in \R^\times$ 
changes $(H,\Delta,\Omega)$ to 
$(\lambda^{12} H,\lambda^{12} \Delta,\lambda^{-1} \Omega)$,
so we may assume that $(A,B)$ lies on the rectangle boundary where $H=1$.
By compactness, the bounds hold on this rectangle boundary
except possibly as $(A,B)$
approaches one of the two corners where $\Delta=0$.
Up to scaling by a $\lambda$ bounded away from $0$ and $\infty$,
these are the curves $\pm y^2=4x^3-g_4(\tau)x -g_6(\tau)$
for $\tau = it$ or $\tau=1/2+it$ as $t \to \infty$
(in the part of the fundamental domain outside a compact set,
these are the $\tau$ 
such that $\Z\tau+\Z$ is homothetic to its complex conjugate).
In these families, each of the Eisenstein series $g_4$ and $g_6$ 
tends to a finite nonzero limit,
so $H$ remains bounded, 
while $|\Delta| \asymp |q| = |e^{2\pi i \tau}| = e^{-2\pi \im \tau}$,
and $\Omega$ is $1$ or $\im \tau$ up to a bounded factor,
so $1 \ll \Omega \ll \log(1/|\Delta|)$.
\end{proof}

\begin{corollary}[\cite{Watkins2008-ExpMath}*{Section~6.2}]
Under the hypotheses of Lemma~\textup{\ref{L:real period}}, we have
$\Omega \ll |\Delta|^{-1/12}$.
\end{corollary}

\begin{proof}
We have $|\Delta| \ll H$.
Then $H^{-1/12} \log(H/|\Delta|) \ll |\Delta|^{-1/12}$
since $x^{1/12} \log(1/x)$ remains bounded as $x \to 0^+$.
\end{proof}

\begin{corollary}[cf.~\cite{Hindry2007}*{Lemma~3.7}]
\label{C:real period}
If $E \in \EE$, then $H^{-1/12} \ll \Omega \ll H^{-1/12} \log H$.
\end{corollary}

\begin{proof}
If $E \in \EE$, then $\Delta$ is a nonzero integer, so $|\Delta| \ge 1$.
Substitute this into Lemma~\ref{L:real period}.
\end{proof}

\begin{remark}
Corollary~\ref{C:real period} is similar to the theorem
relating the naive height to the Faltings height 
\cite{Silverman1986-Faltings-height}*{second statement of Corollary~2.3},
except that the Faltings height is defined using the covolume 
of the period lattice instead of just the real period.
\end{remark}

\begin{remark}
The bounds in Corollary~\ref{C:real period} are best possible, 
up to constants.
For example, for large $a \in \Z_{>0}$,
the curve $y^2=(x-a)(x-a-1)(x+2a+1)$ has $H \asymp a^6$
and $\Omega \asymp a^{-1/2} \log a \asymp H^{-1/12} \log H$;
this shows that the upper bound is sharp.  
\end{remark}

\begin{remark}
\label{R:Neron differential}
If instead of a short Weierstrass model we use the minimal Weierstrass model
$y^2 + a_1 x y + a_3 y = x^3 + a_2 x^2 + a_4 x + a_6$
and a \defi{N\'eron differential} 
$\omega \colonequals \frac{dx}{2 y + a_1 x + a_3}$,
then  $\omega$ differs from $\frac{dx}{2y}$ by bounded powers 
of $2$ and $3$.
So if we define $\Omega$ using the N\'eron differential in place of $\frac{dx}{2y}$,
the estimates in Corollary~\ref{C:real period} are still valid.
It is this $\Omega$ that appears in the
Birch and Swinnerton-Dyer conjecture.
\end{remark}

\begin{remark}
Some authors define the real period as the integral of 
a N\'eron differential over only \emph{one component} of $E(\R)$.
\end{remark}

\subsection{The product of the Tamagawa factors}

Consider $E \in \EE$ of height about~$H$.
Let $\calE$ be the N\'eron model of $E$ over $\Z$.
For each prime $p$,
let $\Phi_p$ be the component group (scheme) 
of the special fiber $\calE_{\F_p}$,
and define the \defi{Tamagawa factor} $c_p \colonequals \#\Phi_p(\F_p)$.

\begin{lemma}
\label{L:product of Tamagawa factors}
We have $\prod_p c_p = H^{o(1)}$.
\end{lemma}

Compare this lemma with work of 
de Weger \cite{deWeger1998}*{Theorem~3}, 
Hindry \cite{Hindry2007}*{Lemma~3.5}, and 
Watkins \cite{Watkins2008-ExpMath}*{pp.~114--115}.

\begin{proof}
For $n \ge 1$, let $\sigma_0(n)$ denote the number of positive divisors of $n$.
Factor the minimal discriminant $\Delta$ of $E$ as $\prod_p p^{e_p}$.
Whenever $e_p>0$, Kodaira and N\'eron proved that 
$c_p \le 4$ or $c_p=e_p$ \cite{SilvermanAEC2009}*{Theorem~VII.6.1},
so in any case $c_p \le (e_p+1)^2 = \sigma_0(p^{e_p})^2$.
Thus
\[
	\prod_p c_p \le \sigma_0(\Delta)^2 = \left( \Delta^{o(1)} \right)^2 = H^{o(1)}.\qedhere
\]
\end{proof}

\begin{remark} \label{R:de Weger}
If instead of $\sigma_0(n) = n^{o(1)}$
we used the more precise bound $\sigma_0(n) \le n^{O(1/\log \log n)}$,
we would get a direct proof of \cite{deWeger1998}*{Theorem~3}, 
which states that $\prod_p c_p \le \Delta^{O(1/\log \log \Delta)}$.
\end{remark}

\subsection{Average size of $L(E,1)$}

The Riemann hypothesis for the $L$-functions $L(E,s)$
would imply the corresponding 
Lindel\"{o}f hypothesis \cite{Iwaniec-Sarnak2000}*{p.~713},
which in turn would imply 
\begin{equation}
\label{E:Lindelof}
	L(E,1) \stackrel{?}\le H^{o(1)}.
\end{equation}
For our calibration, however, 
we need only estimate \emph{averages} of $L(E,1)$, 
so we conjecture the following.

\begin{conjecture}
\label{C:average L(E,1)}
We have $\underset{E \in \EE_{\le H}}{\Average}\; L(E,1) \stackrel{?}= H^{o(1)}$ 
as $H \to \infty$.
\end{conjecture}

In quadratic twist families, the following stronger (unconditional) variant 
of Conjecture~\ref{C:average L(E,1)} is known.

\begin{lemma}
\label{L:average L-value}
Let $E_1$ be an elliptic curve over $\Q$.
Let $E_d$ be its twist by $\Q(\sqrt{d})/\Q$.
Given $D>0$,
let $d$ range over fundamental discriminants satisfying $|d| \le D$.
Then $\underset{|d| \le D}\Average\; L(E_d,1) \asymp 1$
as $D \to \infty$.
\end{lemma}

\begin{proof}
This is a consequence of work of Kohnen and Zagier \cite{Kohnen-Zagier1981}*{Corollaries~5 and~6}.
\end{proof}

\begin{remark}
Lemma~\ref{L:average L-value} makes plausible the conjecture that
\begin{equation} \label{eqn:strongerLE1}
\underset{E \in \EE_{\le H}}{\Average}\; L(E,1) \stackrel{?}\asymp 1 
\end{equation}
as $H \to \infty$.
This conjecture, slightly stronger than Conjecture~\ref{C:average L(E,1)}, 
would not affect the calibration of our heuristic here, 
but is interesting in its own right.  
Young \cite{Young2006-PLMS} proved that \eqref{eqn:strongerLE1} holds 
under the Riemann hypothesis for Dirichlet $L$-functions 
and the equidistribution of the root number of elliptic curves.
\end{remark}  

\subsection{Average size of the Shafarevich--Tate group}
\label{S:average size of Sha}

Let $E \in \EE$.
Define
\[
	\Sha_0(E) \colonequals 
	\begin{cases}
		\#\Sha(E), &\textup{if $\rk E(\Q) = 0$;} \\
		0, &\textup{if $\rk E(\Q) > 0$.}
	\end{cases}
\]
Then the ``rank~$0$ part'' of the Birch and Swinnerton-Dyer conjecture 
states that
\begin{equation} \label{eqn:LE1}
	L(E,1) \stackrel{?}= 
	\frac{\Sha_0 \, \Omega \prod_p c_p}{\# E(\Q)^2_{\tors}};
\end{equation}
see Wiles \cite{Wiles2006} for an exposition
and Stein and Wuthrich \cite{Stein-Wuthrich2013}*{Section~8}
for a summary of some more recent advances towards it.

\begin{theorem}[cf.~\cite{Lang1983-conjectured}*{Conjecture~1}]
\label{T:bound on Sha_0}
Assume the Birch and Swinnerton-Dyer conjecture.  Then the following hold.
\begin{enumerate}[\upshape (a)]
\item 
\label{I:Sha_0 and L(E,1)}
For $E \in \EE$ of height $H$,
we have $\Sha_0(E) = H^{1/12+o(1)} L(E,1)$.
\item 
\label{I:Sha_0}
For $E \in \EE$ of height $H$, 
if the Riemann hypothesis for $L(E,s)$ holds,
then $\Sha_0(E) \leq H^{1/12+o(1)}$. 
\item 
\label{I:average of Sha_0}
If Conjecture~\textup{\ref{C:average L(E,1)}} holds, then $\underset{E \in \EE_{\le H}}{\Average}\; \Sha_0(E) = H^{1/12+o(1)}$ 
as $H \to \infty$.
\end{enumerate}
\end{theorem}

\begin{proof}
By Mazur~\cite{Mazur1977}, we have $\# E(\Q)_{\tors} \le 16$.
By Corollary~\ref{C:real period} and Remark~\ref{R:Neron differential}, 
we have $\Omega = H^{-1/12+o(1)}$.
By Lemma~\ref{L:product of Tamagawa factors},
we have $\prod c_p = H^{o(1)}$.
Substitute all this into \eqref{eqn:LE1} to obtain~\eqref{I:Sha_0 and L(E,1)}.
Combine~\eqref{I:Sha_0 and L(E,1)} 
with \eqref{E:Lindelof} to obtain~\eqref{I:Sha_0}.
Combine~\eqref{I:Sha_0 and L(E,1)} with 
Conjecture~\ref{C:average L(E,1)}
to obtain~\eqref{I:average of Sha_0}.
\end{proof}

\begin{remark}
Theorem~\ref{T:bound on Sha_0}\eqref{I:average of Sha_0} 
agrees with a conjecture of Heath-Brown
and with numerical investigations
of D\k{a}browski, J\k{e}drzejak, and Szymaszkiewicz 
\cite{Dabrowski-Jedrzejak-Szymaszkiewicz2016}*{Section~7}.
\end{remark}

\begin{remark}
\label{R:Waldspurger}
In a family of quadratic twists $E_d$, 
we have $\height E_d \asymp d^6$, 
so Theorem~\ref{T:bound on Sha_0}\eqref{I:Sha_0} would imply
$\Sha_0(E_d) \stackrel{?}\le d^{1/2+o(1)}$ as $d \to +\infty$.
This is consistent with the work of 
Waldspurger \cite{Waldspurger1981}*{Corollaire~2, p.~379} 
relating $\sqrt{\Sha_0(E_d)}$ to the $d$th coefficient $a_d$
of a weight~$3/2$ modular form,
since for such a form we expect $|a_d| \le d^{1/4+o(1)}$.
\end{remark}

%****************************************************************************
\section{The basic model for ranks and Shafarevich--Tate groups}
\label{S:heuristics for ranks}

The construction of the measure $\AA_r$ in Section~\ref{S:distribution of Sha} 
involved alternating matrices that modeled Shafarevich--Tate groups 
of elliptic curves of rank $r$. 
Specifically, the matrices were required to have corank~$r$.
Inspired by this model and interested in the distribution of ranks among all elliptic curves, we propose the following model 
for the arithmetic of an elliptic curve $E$ over $\Q$ of height $H$.
Informally, to each elliptic curve $E$ we will associate a random matrix $A \in \M_n(\Z)_{\alt, \le X}$ such that $\rk(\ker A)$ models $\rk E(\Q)$ and $(\coker A)_{\tors}$ models $\Sha(E)$.
A more precise version of our model depends 
on increasing functions $\nnn(H)$ and $X(H)$ 
to be calibrated later,
with $\nnn(H),X(H) \to \infty$ as $H \to \infty$.  

\subsection{The random model}
\label{S:random model}

We now define a collection of 
independent random variables $(\rk'_E,\Sha'_E)_{E \in \EE}$
taking values in $\Z_{\ge 0} \times \symplectic$.
These random variables will be defined as functions of random matrices, and the only input from the elliptic curve $E$ will be its height.

To define the random variable with index $E$,
let $H \colonequals \height E$,
choose $n$ uniformly at random from ${\Z \intersect [\nnn(H),\nnn(H)+2)}$,
choose $A \in \M_n(\Z)_{\alt, \le X(H)}$ uniformly at random,
define $\rk'_E \colonequals \rk(\ker A)$,
and define $\Sha'_E$ to be $(\coker A)_{\tors}$
equipped with its canonically defined nondegenerate alternating pairing~\cite{Bhargava-Kane-Lenstra-Poonen-Rains2015}*{Sections 3.4 and~3.5}.

\begin{remark}
Replacing $[\nnn(H),\nnn(H)+2)$
with any other interval of length $o(\nnn(H))$ containing $\nnn(H)$
would not affect our results as long as the parity of $n$
becomes equidistributed as $H \to \infty$.
\end{remark}

In the rest of this section, we will prove unconditional theorems about random integral alternating matrices, in particular  
 about the statistical behavior 
of $\rk'_E$ and $\Sha'_E$ as $E$ varies.
These will inform our conjectures about $\rk E(\Q)$ and $\Sha(E)$.

\subsection{First results on random  matrices}

Define the random variable
\[
	\Sha'_{0,E} \colonequals 
	\begin{cases}
		\#\Sha'_E, &\textup{if $\rk'_E = 0$;} \\
		0, &\textup{if $\rk'_E > 0$.}
	\end{cases}
\]
We first prove a theorem about the individual random variables $(\rk'_E,\Sha'_E)$.

\begin{theorem} 
\label{T:random variables for one E}
If the function $X(H)$ grows sufficiently quickly relative to $\nnn(H)$,
then the following hold for $E \in \EE$ 
as $H \colonequals \height E \to \infty$.
\begin{enumerate}[\upshape (a)]
\item \label{I:rank 0 and 1}
\begin{enumerate}[\upshape (1)]
\setcounter{enumii}{-1}
\item 
The probability that $\rk'_E=0$ is $1/2 - o(1)$.
\item 
The probability that $\rk'_E=1$ is $1/2 - o(1)$.
\item 
The probability that $\rk'_E \ge 2$ is $o(1)$.
\end{enumerate}
\item \label{I:probability of Sha'}
Let $r \in \{0,1\}$ and $G \in \symplectic_p$.  
Then 
\[
	\Prob(\Sha'_E[p^{\infty}] \isom G \mid \rk'_E = r) 
	= \Prob_{\DD_{r,p}}(G) + o(1).
\]
(See~\eqref{E:DD_r} for an explicit formula for $\Prob_{\DD_{r,p}}(G)$.)

\item \label{I:global Sha' for rank 1}
 \begin{enumerate}[\upshape (1)]
 \item Let $G \in \symplectic$.  Then
\[
	\Prob(\Sha'_E \isom G \mid \rk'_E = 1) 
	= \prod_p \Prob_{\DD_{1,p}}(G[p^\infty]) + o(1).
\]
 \item More generally, if $\calG \subseteq \symplectic$, then
\[
	\Prob(\Sha'_E \in \calG \mid \rk'_E = 1) 
	= \sum_{G \in \calG} \prod_p \Prob_{\DD_{1,p}}(G[p^\infty]) + o(1).
\]
 \end{enumerate}

\item \label{I:probability of global Sha'}

 \begin{enumerate}[\upshape (1)]
 \item Let $G \in \symplectic$.
       Then $\Prob(\Sha'_E \isom G \mid \rk'_E = 0) = o(1)$.
 \item If $\calG$ is the set of squares of cyclic groups, then 
\[
	\Prob( \Sha'_E \in \calG  \mid \rk'_E = 0)
	= \prod_p \left(1 - \frac{1}{p^2} + \frac{1}{p^3} \right) + o(1).
\]
 \end{enumerate}

\item \label{I:size of Sha'}
 \begin{enumerate}[\upshape (1)]
   \item \label{I:upper bound on Sha'}
   We have $\Sha'_{0,E} \le (X(H)^{\nnn(H)})^{1 + o(1)}$.
   \item \label{I:lower bound on Sha'} 
   The probability that $\Sha'_{0,E} \ge (X(H)^{\nnn(H)})^{1 - o(1)}$ 
   is at least $1/3$.
 \end{enumerate}
\item \label{I:rank at least r}
For fixed $r \ge 1$, 
we have $\Prob(\rk'_E \ge r) = (X(H)^{\nnn(H)})^{-(r-1)/2 + o(1)}$.
\end{enumerate}
\end{theorem}

The proof of~\eqref{I:rank at least r} will require 
the main theorem of Section~\ref{S:counting matrices}, 
while the proofs of \eqref{I:rank 0 and 1}--\eqref{I:size of Sha'} 
are comparatively straightforward 
(although some of them require the Ekedahl sieve).
The constant $1/3$ in \eqref{I:size of Sha'}(2) 
could be improved to any constant less than $1/2$, 
as will be clear from the proof.

\begin{proof}
Let $X \colonequals X(H)$ and $\nnn \colonequals \nnn(H)$.
Any constant depending on $n$ can be assumed to be $X^{o(1)}$
if $X$ grows sufficiently quickly relative to $\nnn$.
\begin{enumerate}[\upshape (a)]
\item 
Since we choose $n$ uniformly in $\Z \intersect [\nnn, \nnn+2)$, 
it is even half of the time and odd half of the time.
Any alternating matrix has even rank,
and a generic alternating matrix of rank $n$ has rank $n$ or $n-1$
according to whether $n$ is even or odd.
As $X \to \infty$ for fixed $n$, 
the probability that an \emph{integer} matrix $A \in \M_n(\Z)_{\alt, \leq X}$ 
has the generic rank tends to~$1$
(it fails on integer points in a proper Zariski-closed subset).
It follows formally that the same holds
if $X$ tends to $\infty$ sufficiently quickly relative to $\nnn$.
Thus $\Prob(\rk'_E=0) = 1/2 - o(1)$ as $H \to \infty$,
and the other statements follow similarly.
\item
For a fixed $k$, and any $n$,
the set $\M_n(\Z)_{\alt, \le X}$ becomes equidistributed in $\M_n(\Z/p^k\Z)$ as $X\rightarrow\infty$.
When $n$ is even, the same holds for the subset of 
$A \in \M_n(\Z)_{\alt,\le X}$ 
satisfying $\dim(\ker A)=0$ 
(these are the ones in a nonempty Zariski-open subset).
Given $G$, there exists a positive integer $k$ such that 
the condition $\Sha'_E[p^\infty] \isom G$ depends only on $A$ modulo $p^k$.  
Thus, if $X$ is sufficiently large relative to $n$, 
then 
\[
	\Prob(\Sha'_E[p^\infty] \isom G \mid \rk'_E = 0) 
	= \Prob_{\AA_0}(G) + o(1)
\]
as $H \to \infty$.
By Theorem~\ref{T:three same}, $\AA_0$ coincides with $\DD_0 = \DD_{0,p}$.
An analogous argument applies if we condition on $\rk'_E$ being~$1$.

\item 
(1) We apply the Ekedahl sieve as adapted by Poonen and Stoll
in \cite{Poonen-Stoll1999}*{Section~9.3}.
Consider a large odd integer $n$.
Let $U_\infty = \M_n(\R)_{\alt}$.
For each prime $p$, 
let $U_p$ be the set of $A \in \M_n(\Z_p)_{\alt}$
such that $(\coker A)_{\tors}[p^\infty] \not\isom G[p^\infty]$.
Let $s_p$ be the Haar measure of $U_p$.
The image of $U_p$ in $\M_n(\F_p)_{\alt}$ is contained
in the set of $\F_p$-points of the subscheme of $\Aff^{n^2}_{\Z}$
parametrizing matrices of corank~$\ge 3$,
and this implies that $s_p = O(1/p^2)$ uniformly in $n$.
Now, \cite{Poonen-Stoll1999}*{Lemma~21}
implies that hypothesis~(10) in \cite{Poonen-Stoll1999}*{Lemma~20} holds.
The conclusion of \cite{Poonen-Stoll1999}*{Lemma~20} 
for $S \colonequals \emptyset$
implies that the density of $A \in \M_n(\Z)_{\alt}$
satisfying $(\coker A)_{\tors} \isom G$ 
equals $\prod_p (1-s_p)$.
Because of the uniform estimate on $s_p$,
we may take the limit as $n \to \infty$ \emph{inside} the product,
in which case $1-s_p$ tends to $\Prob_{\DD_{1,p}}(G[p^\infty])$
by Theorem~\ref{T:three same},
so (1) follows.\\
(2) This follows formally from~(1) and the fact that
$\sum_{G \in \symplectic} \prod_p \Prob_{\DD_{1,p}}(G[p^\infty]) = 1$.

\item 
(1)
Since $\rk'_E=0$, we have $\# \Sha'_E=\abs{\det A}$.
By the same reasoning as in the proof of~\eqref{I:rank 0 and 1}, 
the probability that an integer matrix $A \in \M_n(\Z)_{\alt, \leq X}$ 
has $\abs{\det A}$ equal to a fixed value tends to~$0$ 
if $X$ tends to $\infty$ sufficiently quickly relative to $\nnn$.\\
(2) This follows from the Ekedahl sieve 
as in the proof of \eqref{I:global Sha' for rank 1}(1) above, 
since the ``square of cyclic''
condition can be checked on the $p$-primary part one $p$ at a time,
and for each $p$ the reductions modulo $p$
of the $A \in \M_n(\Z)_{\alt}$ such that $(\coker A)[p^\infty]$ 
is \emph{not} cyclic lie in the $\F_p$-points of a subscheme
of $\Aff^{n^2}_\Z$ of codimension $\ge 2$.

\item 
By~\eqref{I:rank 0 and 1},
we have $\Prob(\rk'_E = 0) = 1/2-o(1)$.
If $\rk'_E > 0$, then $\Sha'_{0,E}=0$.
If $\rk'_E = 0$, then $\Sha'_{0,E}$ 
is the absolute value of the determinant of a random
$A \in \M_n(\Z)_{\alt,\le X}$,
which is the absolute value of a degree~$n$ polynomial 
evaluated on a box of dimensions very large relative to $n$.
This implies that there are constants $m_n,M_n>0$ depending only on $n$
such that $\Sha'_{0,E} \le M_n X^n$
and such that $\Prob(\Sha'_{0,E} \ge m_n X^n | \rk'_E=0)$ is at least $9/10$.
Since $(1/2-o(1))(9/10) > 1/3$ and $X^{n + o(1)} = (X^{\nnn})^{1+o(1)}$,
the results follow.
\item
We have
\begin{align*}
	\Prob(\rk(\ker A)\geq r \mid n \equiv r \!\!\!\! \pmod{2})
	&=
	\frac{\#\{A \in \M_n(\Z)_{\alt,\le X} : \rk(\ker A) \ge r\}}
	{\# \M_n(\Z)_{\alt,\le X}} \\
	&= \frac{X^{n(n-r)/2 + o(1)}}{X^{n(n-1)/2 + o(1)}} 
	\quad\textup{(by Theorem~\ref{thm:EskinKatznelsonAlternating})} \\
	&= (X^n)^{-(r-1)/2 + o(1)} \\
	&= (X^\nnn)^{-(r-1)/2 + o(1)} \;\textup{(since $n = \nnn + o(\nnn) = \nnn(1+o(1))$),}
\end{align*}
as $H\ra\infty$.   
On the other hand,
\begin{align*}
	\Prob(\rk(\ker A)\geq r \mid n \not\equiv r \!\!\!\! \pmod{2})
	&= \Prob(\rk(\ker A)\geq r+1 \mid n \not\equiv r \!\!\!\! \pmod{2})\\
	&= (X^\nnn)^{-r/2 + o(1)},
\end{align*}
by a similar calculation.
Combining these yields the result.
\qedhere
\end{enumerate}
\end{proof}

\begin{remark}\label{R:Stanley-Wang}

See \cite{Wang-Stanley2017} 
for related work applying the Ekedahl sieve to study cokernels of 
not-necessarily-alternating integral matrices.
\end{remark}

\begin{remark}\label{R:robustcok}
The conclusion of 
Theorem~\ref{T:random variables for one E}\eqref{I:probability of Sha'} 
is likely robust.
The analogous conclusion for symmetric matrices is proved in \cite{Wood2017} 
without requiring any kind of uniform distribution of matrix entries.
For instance, it holds even if $X(H)$ is always $1$.
\end{remark}

\begin{remark}\label{R: high rank Sha}
The proof of~\eqref{I:probability of Sha'} implicitly used that 
the $\Z$-points on the moduli space of matrices (isomorphic
to $\Aff^{n^2}$) are equidistributed in the $\Z_p$-points.
For $r \ge 2$, this affine space gets replaced by a subvariety $V$
defined by the vanishing of certain minors,
and it is not clear that $V(\Z)$ is equidistributed in $V(\Z_p)$.
In fact, heuristics inspired by the circle method suggest that this
might be false, and numerical experiments also suggest this.
In this case, perhaps the three conjectural answers 
to Question~\ref{Q:distribution of Sha} are wrong for $r \ge 2$.
In particular, perhaps the ``canonical probability measure'' 
from \cite{Bhargava-Kane-Lenstra-Poonen-Rains2015}*{Section~2}
on the set
\[ 
	\{A \in \M_n(\Z_p)_{\alt} : \rk_{\Z_p}(\ker A) = r \} 
\]
used to define $\AA_r$ (the measure proportional to $p$-adic volume)
should be replaced by the measure that reflects 
the density of integer points.
\end{remark}

Next we will pass from Theorem~\ref{T:random variables for one E},
which concerns the random variable associated to one $E$,
to Corollary~\ref{C:random variables for many E},
which concerns the aggregate behavior of the random variables
associated to all $E \in \EE_{\le H}$, as $H \to \infty$.
To do this we will apply the following standard result, 
a version of the law of large numbers 
in which the random variables do not have to be identically distributed.

\begin{lemma}[Theorem~2.3.8 of \cite{Durrett2010}]
\label{L:law of large numbers} 
Let $B_1,B_2,\ldots$ be a sequence of independent events.
For $i \ge 1$, let $p_i$ be the probability of $B_i$.
If $\sum p_m$ diverges, then with probability~$1$, 
\[
	\#\{i \le m : B_i \textup{ occurs}\} 
	= (1+o(1)) \sum_{i=1}^m p_i
\] 
as $m \to \infty$.
\end{lemma}

\begin{corollary} 
\label{C:random variables for many E}
If $X(H)$ grows sufficiently quickly relative to $\nnn(H)$, then the following hold with probability~$1$.
\begin{enumerate}[\upshape (a)]
\item \label{I:distribution of rk'_E} 
We have
\begin{center}
$\ProbE(\{E  : \rk'_E = 0\})=\ProbE(\{E  : \rk'_E = 1\})=1/2$ and \\
$\ProbE(\{E  : \rk'_E \geq 2\})=0$.
\end{center}
\item\label{I:distribution of Sha'} 
For each $r \in \{0,1\}$ and $G \in \symplectic_p$,
\[
	\ProbE(\{E : \Sha'_E[p^\infty] \isom G \} \mid \rk'_E = r) 
	= \Prob_{\DD_{r,p}}(G).
\]

\item \label{I:global Sha' distribution for rank 1}
 \begin{enumerate}[\upshape (1)]
 \item For each $G \in \symplectic$, we have
\[
	\ProbE(\{E : \Sha'_E \isom G \} \mid \rk'_E = 1) 
	= \prod_p \Prob_{\DD_{1,p}}(G[p^\infty]).
\]
 \item More generally, for each $\calG \subseteq \symplectic$, we have
\[
	\ProbE(\{E : \Sha'_E \in \calG \} \mid \rk'_E = 1) 
	= \sum_{G \in \calG} \prod_p \Prob_{\DD_{1,p}}(G[p^\infty]).
\]
 \end{enumerate}

\item\label{I:distribution of global Sha'} 
 \begin{enumerate}[\upshape (1)]
 \item
For each $G \in \symplectic$, we have 
$
	\ProbE(\{E : \Sha'_E \isom G \} \mid \rk'_E = 0) 
	= 0.
$ 
\item If $\calG$ is the set of squares of cyclic groups,
then 
\[
\ProbE(\{E : \Sha'_E \in \calG \} \mid \rk'_E = 0)	= \prod_p \left(1 - \frac{1}{p^2} + \frac{1}{p^3} \right).
\]

\end{enumerate}
 
\item \label{I:average Sha'}
We have
\[ 
	\Average_{E \in \EE_{\le H}} \Sha'_{0,E} = (X(H)^{\nnn(H)})^{1 + o(1)}
\]
as $H \to \infty$, assuming that the function $f(H) \colonequals X(H)^{\nnn(H)}$
satisfies $f(2H) \le f(H)^{1+o(1)}$. 
\end{enumerate}
\end{corollary}

\begin{proof}\hfill
For $E \in \EE$, let $B_E$ be the event $\rk'_E=r$,
and let $C_E$ be the event that $\rk'_E=r$ and $\Sha'_E[p^{\infty}] \isom G$.
\begin{enumerate}[\upshape (a)]
\item
Apply Lemma~\ref{L:law of large numbers} to $(B_E)$,
and use Theorem~\ref{T:random variables for one E}\eqref{I:rank 0 and 1}.
\item
Apply Lemma~\ref{L:law of large numbers} to $(B_E)$ and $(C_E)$,
and use Theorem~\ref{T:random variables for one E}(\ref{I:rank 0 and 1},\ref{I:probability of Sha'})
to compute the denominator and numerator in the definition of $\ProbE$.
\item 
Again apply Lemma~\ref{L:law of large numbers},
and use 
Theorem~\ref{T:random variables for one E}\eqref{I:global Sha' for rank 1}.

\item 
\begin{enumerate}[\upshape (1)]
 \item Apply Lemma~\ref{L:law of large numbers} 
 to the event $\rk'_E=0$ and $\Sha'_E \not\isom G$, and use 
 Theorem~\ref{T:random variables for one E}\eqref{I:probability of global Sha'}(1).
 \item 
 Apply Lemma~\ref{L:law of large numbers} 
 to the event $\rk'_E=0$ and $\Sha'_E \in \calG$, and use 
 Theorem~\ref{T:random variables for one E}\eqref{I:probability of global Sha'}(2).
\end{enumerate}

\item 
By Theorem~\ref{T:random variables for one E}\eqref{I:size of Sha'}(1), 
\[
	\Average_{E \in \EE_{\le H}} \Sha'_{0,E} 
	\le \max_{E \in \EE_{\le H}} \Sha'_{0,E} 
	\le (X(H)^{\nnn(H)})^{1 + o(1)}
\]
as $H \to \infty$.
By Theorem~\ref{T:random variables for one E}\eqref{I:size of Sha'}(2) 
and the law of large numbers, 
with probability~$1$,
as $H \to \infty$,
at least $1/4$ of the elliptic curves with height in $(H/2,H]$
satisfy $\Sha'_{0,E} \ge (X(\height E)^{\nnn(\height E)})^{1 - o(1)}$.
We have $(X(\height E)^{\nnn(\height E)})^{1 - o(1)}=(X(H)^{\nnn(H)})^{1-o(1)}$
by the growth hypothesis on $f(H)$.
Since a positive fraction of the elliptic curves in $\EE_{\le H}$
have height in $(H/2,H]$,
this implies 
$\Average_{E \in \EE_{\le H}} \Sha'_{0,E} \ge (X(H)^{\nnn(H)})^{1-o(1)}$.\qedhere
\end{enumerate}
\end{proof}

\subsection{Consequences for coranks of random matrices}

Comparing 
Theorem~\ref{T:bound on Sha_0}\eqref{I:average of Sha_0}
and Corollary~\ref{C:random variables for many E}\eqref{I:average Sha'}
suggests choosing $X(H)$ and $\nnn(H)$ so that
\begin{equation}
\label{E:calibration}
	X(H)^{\nnn(H)} = H^{1/12+o(1)}
\end{equation}
as $H \to \infty$.

\begin{remark}\label{R:matchupper}
Alternatively, matching the conditional upper bound of 
Theorem~\ref{T:bound on Sha_0}\eqref{I:Sha_0}
with an upper bound for $\det A$ 
would have also suggested~\eqref{E:calibration}.
\end{remark}

We now prove a theorem about the asymptotic aggregate behavior of the $\rk'_E$.

\begin{theorem}
\label{T:rank 21}
If $\nnn(H)$ grows sufficiently slowly relative to $H$,
and $X(H)^{\nnn(H)} = H^{1/12+o(1)}$,
then the following hold with probability~$1$:
\begin{enumerate}[\upshape (a)]
\item \label{I:21 theorem}
All but finitely many $E \in \EE$ satisfy $\rk'_E \le 21$.
\item \label{I:1 to 20}
For $1 \le r \le 20$, 
we have $\#\{ E \in \EE_{\le H} : \rk'_E \ge r \} = H^{(21-r)/24+o(1)}$.
\item \label{I:21}
We have $\#\{ E \in \EE_{\le H} : \rk'_E \ge 21 \} \le H^{o(1)}$.
\end{enumerate}
\end{theorem}

\begin{proof}
Fix $r \ge 1$.
For $E \in \EE$, let $p_{E,r} \colonequals \Prob(\rk'_E \ge r)$.
By Theorem~\ref{T:random variables for one E}\eqref{I:rank at least r},
if $E$ is of height~$H$, then 
\[
	p_{E,r} = (X(H)^{\nnn(H)})^{-(r-1)/2 + o(1)} = H^{-(r-1)/24 + o(1)}.
\]
It follows that 
\begin{align*}
	\sum_{E \in \EE_{\le H}} p_{E,r}
	&= \sum_{E \in \EE_{\le H}} (\height E)^{-(r-1)/24 + o(1)} \\
	&= \begin{cases}
		H^{(21-r)/24 + o(1)}, & \textup{ if $1 \le r \le 21$;} \\
		O(1), & \textup{ if $r>21$,} \\
	\end{cases}
\end{align*}
by summing over dyadic intervals,
using the estimate $\# \EE_{\le H} = (\kappa + o(1)) H^{20/24}$ 
from~\eqref{eqn:height}.

If $\sum_{E \in \EE} p_{E,r}$ converges, 
as happens for $r>21$ and possibly also for $r=21$,
then the Borel--Cantelli lemma  implies that 
$\{E \in \EE: \rk'_E \ge r\}$ is finite.
If $\sum_{E \in \EE} p_{E,r}$ diverges, 
as happens for $1 \le r \le 20$ and possibly also for $r=21$,
then Lemma~\ref{L:law of large numbers} yields
\[
	\#\{E \in \EE_{\le H}: \rk'_E \ge r\}
	= (1+o(1))\sum_{E \in \EE_{\le H}} p_{E,r} = H^{(21-r)/24 + o(1)}.\qedhere
\]
\end{proof}

\begin{remark}
The conclusion that $\rk'_E$ is uniformly bounded with probability~$1$ 
is robust.  For example, 
if the assumption $X(H)^{\nnn(H)} = H^{1/12+o(1)}$ in \eqref{E:calibration}
is replaced by $X(H)^{\nnn(H)} = H^{c+o(1)}$ 
for a different positive constant~$c$,
then the same conclusion follows, 
but the bound beyond which there are only finitely many $E$ 
might no longer be $21$.
Another example: taking our matrix coefficients in a sphere instead of a box 
(as we actually do in Section~\ref{S:counting matrices}) 
does not change Theorem~\ref{thm:EskinKatznelsonAlternating},
so it does not change Theorem~\ref{T:rank 21} either.
\end{remark}

\begin{remark}
Although it would have been nice to have specifications 
for $X(H)$ and $\nnn(H)$ individually,
the specification of $X(H)^{\nnn(H)}$ alone 
sufficed for Theorem~\ref{T:rank 21}.
\end{remark}

\section{Predictions for elliptic curves}
\label{S:consequences}

The results of Section~\ref{S:heuristics for ranks} 
are unconditional theorems about random matrices.  
We now conjecture that some statements 
about the statistics of $(\rk'_E,\Sha'_E)$ as $E$ varies
are true also for the \emph{actual} $(\rk E(\Q),\Sha(E))$.

\subsection{Theoretical evidence}
\label{S:theoretical evidence}

Several consequences of this heuristic are widely believed conjectures for elliptic curves.
\begin{enumerate}[(i)]
\item \label{I:minimalist conjecture}
Corollary~\ref{C:random variables for many E}\eqref{I:distribution of rk'_E} suggests the ``minimalist conjecture'' 
that ranks of elliptic curves are $0$ half the time and $1$ half the time
asymptotically,
as has been conjectured by others for quadratic twist families,
including Goldfeld \cite{Goldfeld1979}*{Conjecture~B} 
and Katz and Sarnak \cites{Katz-Sarnak1999a,Katz-Sarnak1999b}.
\item 
Corollary~\ref{C:random variables for many E}\eqref{I:distribution of Sha'} 
predicts the distribution of $\Sha(E)[p^\infty]$ 
for elliptic curves $E$ of rank $r$ in the cases $r=0,1$.  
This distribution agrees with the three distributions 
in Section~\ref{S:distribution of Sha}.  
More generally, for any finite set $S$ of primes,
our model predicts the joint distribution of $\Sha(E)[p^\infty]$ 
for $p \in S$ as $E$ varies among rank $r$ curves in the cases $r=0,1$, 
and these predictions agree with the conjectures in 
\cites{Delaunay2001,Delaunay2007,Delaunay-Jouhet2014a} 
and \cite{Bhargava-Kane-Lenstra-Poonen-Rains2015}*{Section~5.6}.
\item 
Corollary~\ref{C:random variables for many E}\eqref{I:global Sha' distribution for rank 1} predicts that $\Sha(E)$ for $E$ of rank~$1$
is distributed according to the conjectures in 
\cites{Delaunay2001,Delaunay2007,Delaunay-Jouhet2014a}.
\item \label{I:consequence rank 0 has large Sha}
Given $G \in \symplectic$, 
Corollary~\ref{C:random variables for many E}\eqref{I:distribution of global Sha'} predicts that 
the density of rank $0$ curves $E$ with $\Sha(E) \isom G$ is zero, 
in agreement with the conjectures discussed 
in Remark~\ref{R:large and small Sha}, and that the density of rank $0$ curves $E$ with
$\Sha(E)$  the square of a cyclic group is the density conjectured by Delaunay \cite{Delaunay2001}*{Example~E}.

\end{enumerate}

\subsection{Predictions for ranks}
\label{S:predictions for ranks}

Our heuristic predicts also that the three conclusions of 
Theorem~\ref{T:rank 21} hold 
if $\rk'_E$ is replaced by $\rk E(\Q)$.
Specifically, it predicts
\begin{enumerate}[\upshape (a)]
\item \label{I:le 21} 
All but finitely many $E \in \EE$ satisfy $\rk E(\Q) \le 21$.
\item \label{I:rank count} 
For $1 \le r \le 20$, we have
${\#\{E\in \EE_H \mid \rk E(\Q) \geq r\}}\stackrel{?}=H^{(21-r)/24+o(1)}$.
\item We have $\#\{ E \in \EE_{\le H} : \rk E(\Q) \ge 21 \} 
	\stackrel{?}\le H^{o(1)}$.
\end{enumerate}
In particular, \eqref{I:le 21} would imply that 
ranks of elliptic curves over $\Q$ are bounded.

The prediction~\eqref{I:rank count} for $r=1$ was discussed 
in~\eqref{I:minimalist conjecture} in Section~\ref{S:theoretical evidence}.
The prediction of \eqref{I:rank count} for $r=2$ 
is consistent with all of the previous conjectures 
in Section~\ref{S:density of rank 2}.
See Section~\ref{S:density of rank 3} for a comparison of 
the prediction of \eqref{I:rank count} for $r=3$ 
to the other conjectures for this asymptotic.  

There are many other predictions made by the model of Section~\ref{S:heuristics for ranks}, though in some cases we are prevented from giving them explicitly because we do not know the corresponding fact about counting alternating matrices.  We mention some of these in Section~\ref{S:further}.

\subsection{Other families of elliptic curves}
Instead of taking all elliptic curves over $\Q$,
one could restrict to other families of elliptic curves, 
such as a family of quadratic twists 
or a family with prescribed torsion subgroup, 
and prove analogues of Theorem~\ref{T:rank 21}.
The predictions given by such an analogue are summarized 
in the following two tables;
under ``rank bound'' is an integer such that our model predicts
that the family contains only finitely many elliptic curves 
of strictly higher rank.

First, we consider a family of twists.
In some cases, these predictions are stronger than the predictions coming from 
variants of Granville's heuristic~\cite{Watkins-et-al2014}*{Section~11}.

\begin{center}
\begin{tabular}{c|c|c|c}
& $\#$ curves of height $\le H$ & rank bound & Granville \\ \hline
quadratic twists of a fixed $E_0$ & $H^{1/6} = H^{4/24}$ & $5$ & $7$ \\
quartic twists of $y^2=x^3-x$ & $H^{1/3} = H^{8/24}$ & $9$ & $11$ \\
sextic twists of $y^2=x^3-1$ & $H^{1/2} = H^{12/24}$ & $13$ & $13$ \\
all elliptic curves & $H^{5/6} = H^{20/24}$ & $21$ & $21$ \\
\end{tabular}
\end{center}

Next, we consider a family with prescribed torsion.
Harron and Snowden \cite{HarronSnowden2017}
prove that for each finite abelian group $T$ that arises, 
$\#\{E \in \EE_{\le H} : E(\Q)_{\tors} \isom T\} \asymp H^{1/d}$ 
for some $d \in \Q_{>0}$ depending on $T$.
For such a family, 
our heuristic suggests that the expression
$b_T \colonequals \limsup\,\{\rk E(\Q) : E(\Q)_{\tors} \isom T \}$
is bounded above by $1+\lfloor{24/d}\rfloor$. 
On the other hand, explicit families provide 
lower bounds on $b_T$ \cite{Dujellatorsgeneric}.
These upper and lower bounds are given in 
the last two columns of the following table.
Remarkably, for each $T$, the bounds are close and the 
conjectured upper bound is at least as large as the proven lower bound.

\begin{center}
\begin{tabular}{c|c|c|c}
torsion subgroup & $\#$ curves of height $\le H$ & rank bound & known lower bound \\ \hline
--- & $H^{5/6}$ & $21$ & $19$ \\
$\Z/2\Z$ & $H^{1/2}$ & $13$ & $11$ \\
$\Z/3\Z$ & $H^{1/3}$ & $9$ & $7$ \\
$\Z/4\Z$ & $H^{1/4}$ & $7$ & $6$ \\
$\Z/5\Z$ & $H^{1/6}$ & $5$ & $4$ \\
$\Z/6\Z$ & $H^{1/6}$ & $5$ & $5$ \\
$\Z/7\Z$ & $H^{1/12}$ & $3$ & $2$ \\
$\Z/8\Z$ & $H^{1/12}$ & $3$ & $3$ \\
$\Z/9\Z$ & $H^{1/18}$ & $2$ & $1$ \\
$\Z/10\Z$ & $H^{1/18}$ & $2$ & $1$ \\
$\Z/12\Z$ & $H^{1/24}$ & $2$ & $1$ \\
$\Z/2\Z \times \Z/2\Z$ & $H^{1/3}$ & $9$ & $8$ \\
$\Z/2\Z \times \Z/4\Z$ & $H^{1/6}$ & $5$ & $5$ \\
$\Z/2\Z \times \Z/6\Z$ & $H^{1/12}$ & $3$ & $3$ \\
$\Z/2\Z \times \Z/8\Z$ & $H^{1/24}$ & $2$ & $1$ 
\end{tabular}
\end{center}

%****************************************************************************
\section{Counting alternating matrices of prescribed rank}
\label{S:counting matrices}

In this section we prove Theorem~\ref{thm:EskinKatznelsonAlternating},
which was used in the proof of 
Theorem~\ref{T:random variables for one E}\eqref{I:rank at least r}.

\subsection{Statement and overview of proof}
\label{S:counting matrices statement}

\begin{theorem} \label{thm:EskinKatznelsonAlternating}
If  $1 \leq r \leq n$ and $n-r$ is even, then 
\[
	\#\{A \in \M_n(\Z)_{\alt,\leq X} : \rk(\ker A ) \ge r\}
	\asymp_n X^{n(n-r)/2}.
\]
\end{theorem}

In fact, we will prove the same asymptotic for the count with $\ge r$
replaced by $=r$; then Theorem~\ref{thm:EskinKatznelsonAlternating}
follows by summing.
Also, the $\ell^\infty$-norm of a matrix $A=(a_{ij})$ 
is bounded above and below by
the $\ell^2$-norm times constants depending on $n$,
so we may instead use the $\ell^2$-norm, 
which is defined by $|A|^2=\sum_{i,j} a_{ij}^2$.
Finally, we may use $\rk(\ker A) = n - \rank A$, and rename $r$ as $n-r$.
This leads us to define 
\[
	N_{n,r}(T) \colonequals 
	\#\{A \in \M_n(\Z)_{\alt}:  \rank A = r \textup{ and } |A|<T\}.
\]
Now the result to be proved is as follows.

\begin{theorem}
\label{T:counting matrices by l^2}
If $0 \le r \le n-1$ and $r$ is even, then $N_{n,r}(T) \asymp_n T^{nr/2}$.
\end{theorem}

Theorem~\ref{T:counting matrices by l^2} 
is the analogue for alternating matrices 
of the Eskin--Katznelson theorem 
counting integral \emph{symmetric} matrices of specified rank 
\cite{Eskin-Katznelson1995}*{Theorem~1.2}.
Our proof of Theorem~\ref{T:counting matrices by l^2} 
follows the Eskin--Katznelson proof closely;
indeed, the differences are almost entirely numerical
(though we have incorporated a few simplifications,
notably in the proofs of Lemma~\ref{L: changeofbasis}
and Theorem~\ref{T:mainlowerbound}).
The strategy is to count the matrices
(the rank~$r$ matrices $A \in \M_n(\Z)_{\alt}$ with $|A|<T$)
by grouping them 
according to which primitive rank~$r$ sublattice $\bL \subseteq \Z^n$
contains the rows of $A$.

To obtain an upper bound on $N_{n,r}(T)$,
first, we bound the determinant of the lattices $\bL$ that arise 
in the previous sentence 
(Corollary~\ref{C: suffcondition}).
Second, a theorem of Schmidt (Theorem~\ref{T:Schmidt}) 
bounds the number of $\bL$ of bounded determinant.
Third, the matrices $A$ associated to \emph{one} such $\bL$
correspond to certain points in the lattice $\calA(\bL)$ of matrices
whose rows are contained in $\bL$,
so they can be counted by another result
of Schmidt on lattice points in a growing ball (Lemma~\ref{L:countinlattice}).

To obtain a matching lower bound on $N_{n,r}(T)$, 
it will turn out that 
it suffices to count lattices having a basis
consisting of ``almost orthogonal'' vectors of roughly comparable length 
(Theorem~\ref{T:mainlowerbound}). 

\subsection{Lattices}

Fix $n \ge 0$, 
and let $\parenthesispairing$ and $\ltwonorm$ denote the standard 
inner product and $\ell^2$-norm on $\R^n$.
By a \defi{lattice} in $\R^n$, 
we mean a discrete subgroup $\bL \subset \R^n$;
its rank $r$ might be less than $n$.
By convention, each $\Z$-basis $\{\ell_1,\ldots,\ell_r\}$ of $\bL$
is ordered so that $|\ell_1| \le \cdots \le |\ell_r|$.
A lattice $\bL \subset \Z^n$ is \defi{primitive} 
if it is not properly contained in any other sublattice of $\Z^n$ 
of the same rank.
The \defi{determinant} $d(\bL) \in \R_{>0}$ of $\bL$ 
is the $r$-dimensional volume
(with respect to the metric induced by $\ltwonorm$)
of the parallelepiped spanned by any $\Z$-basis $\{\ell_1,\ldots,\ell_r\}$ 
of $\bL$; then $d(\bL)^2 = \det (\ell_i,\ell_j)_{1 \le i,j \le r}$.
Among all bases $\{\ell_1,\ldots,\ell_r\}$ for $\Lambda$,
any one that minimizes the product $|\ell_1|\cdots|\ell_r|$ is called 
a \defi{reduced basis}.
(This is equivalent to the usual definition of Minkowski reduced basis.)

\begin{theorem}[Minkowski]
\label{T:Minkowski}
If $\{\ell_1, \ldots, \ell_r\}$ is a reduced basis for $\bL$, then
$d(\bL) \asymp_r |\ell_1| \cdots |\ell_r|$.
\end{theorem}

Theorem~\ref{T:Minkowski} can be interpreted as saying 
that a reduced basis is ``almost orthogonal''.
The following lemma is essentially a reformulation of Minkowski's theorem 
\cite{Eskin-Katznelson1995}*{Lemma~2.1, comment after Lemma~2.2}.

\begin{lemma}
\label{L: triangle}
Fix $r$ and a positive constant $C$.
Let $\bL$ be a lattice with basis $\{u_1,\ldots,u_r\}$
satisfying $d(\bL) \geq C |u_1|\cdots|u_r|$.
Then for any $a_1,\ldots,a_r \in \R$, 
\[
	\left|\sum_{j=1}^r a_j u_j\right|^2 
	\; \gg_{r,C} \; 
	\sum_{j=1}^r a_j^2 |u_j|^2.
\]
\end{lemma}

The following lemma shows that different choices of reduced bases, 
or even bases within a constant factor of being reduced, 
have very similar lengths.

\begin{lemma}[Lemma~2.2 of \cite{Eskin-Katznelson1995}]
\label{L: uniformity}
Fix $r$ and a positive constant $C$.
Let $\bL$ be a rank~$r$ lattice 
with reduced basis $\{\ell_1, \ldots, \ell_r\}$.
If $\{u_1,\ldots,u_r\}$ is another basis of $\bL$,
and $d(\bL) \geq C|u_1| \cdots |u_r|$,
then $|u_j| \asymp_{r,C} |\ell_j|$ for all $j$.
\end{lemma}

For $T \in \R_{>0}$, 
define the ball $B(T) \colonequals \{x \in \R^n : |x| < T\}$.
For any lattice $\bL \subset \R^n$, let 
\[
	N(T,\bL) \colonequals \#(\bL \intersect B(T)).
\]
Let $\calV_r$ denote the volume of the $r$-dimensional unit ball.

\begin{lemma}[\cite{Schmidt1968}*{Lemma~2}]
\label{L:countinlattice}
Let $\bL$ be a rank~$r$ lattice with reduced basis $\{\ell_1,\ldots,\ell_r\}$. 
Then 
\[
	N(T, \bL) = \frac{\calV_r T^r}{d(\bL)} + 
	O_r\left(\sum_{j=0}^{r-1} \frac{T^j}{|\ell_1|\cdots|\ell_j|}\right).
\]
\end{lemma}

Let $P_{n,r}(t)$ denote the number of primitive rank~$r$ sublattices of $\Z^n$
of determinant at most $t$.
The following is a crude version of a more precise theorem of Schmidt.
\begin{theorem}[\cite{Schmidt1968}*{Theorem~1}]
\label{T:Schmidt}
If $1 \le r \le n-1$, then $P_{n,r}(t) \asymp_n t^n$.
\end{theorem}

\subsection{Lattices of alternating matrices}

{}From now on, $\bL$ is a primitive rank~$r$ lattice in $\Z^n$,
and $\{\ell_1,\ldots,\ell_r\}$ is a reduced basis of $\bL$. 
Define
\[
	\calA(\bL) \colonequals 
	\{A \in \M_n(\R)_{\alt}:  \textup{every row of $A$ is in $\bL$}\}.
\]
View $\calA(\bL)$ as a lattice in the space $\M_n(\R)\isom\R^{n^2}$.
If $A \in \calA(\bL)$, then $\rank A \leq r$;
we write
\[
	N(T, \calA(\bL)) = N_1(T, \calA(\bL)) + N_2(T, \calA(\bL)),
\]
where $N_1$ counts the matrices of rank exactly $r$ 
and $N_2$ counts those of rank $<r$.
If $A \in \M_n(\Z)_{\alt}$ and $\rank A = r$, 
then there exists a unique primitive rank~$r$ lattice $\bL \subset \Z^n$ 
such that $A \in \calA(\bL)$;
namely, $\bL = (\textup{row space of $A$}) \intersect \Z^n$.
Thus, as in \cite{Eskin-Katznelson1995}*{Proposition~1.1},
\begin{equation}\label{E:countbyL}
	N_{n,r}(T) = \sumprime_{\rk \bL = r} N_1(T, \calA(\bL)),
\end{equation}
where the prime indicates that $\bL$ ranges over 
\emph{primitive} rank~$r$ lattices in $\Z^n$.
We will use this to estimate $N_{n,r}(T)$.

To apply Lemma~\ref{L:countinlattice} to count points in $\calA(\bL)$, 
we need good estimates on the size of a reduced basis of $\calA(\bL)$.  
Identify $\R^n \tensor \R^n$ with $\M_n(\R)$ by mapping $u \tensor v$ to $uv^T$.
For $1 \le i < j \le r$, define
\[
	R_{ij} := \ell_i \tensor \ell_j - \ell_j \tensor \ell_i \in \M_n(\R).
\]

\begin{lemma}[Analogue of Lemma~3.2 of \cite{Eskin-Katznelson1995}]
\label{L: absbound}
\hfill
\begin{enumerate}[\upshape (a)]
\item \label{ijst}
If $i<j$ and $s<t$, then $(R_{ij}, R_{st}) 
= 2 (\ell_i,\ell_s) (\ell_j,\ell_t) - 2 (\ell_i,\ell_t) (\ell_j,\ell_s)$.
\item \label{|R_ij|}
For $1 \leq i < j \leq r$, we have $|R_{ij}| \asymp_r |\ell_i| |\ell_j|$.
\end{enumerate}
\end{lemma}

\begin{proof}
\hfill
\begin{enumerate}[\upshape (a)]
\item 
Distribute and use the identity 
$(u \tensor v, u' \tensor v') = (u, u')(v, v')$ four times.
\item 
Taking $(i,j)=(s,t)$ in~\eqref{ijst} yields
\[
|R_{ij}|^2 = 2|\ell_i|^2|\ell_j|^2 - 2(\ell_i, \ell_j)^2 = 2 |\ell_i|^2 |\ell_j|^2 \sin^2 \theta,
\]
where $\theta$ is the angle between $\ell_i$ and $\ell_j$.
By Theorem~\ref{T:Minkowski}, we have $\sin \theta \asymp_r 1$.\qedhere
\end{enumerate}
\end{proof}

\begin{proposition}[Analogue of Proposition~3.3 of \cite{Eskin-Katznelson1995}] 
\label{P: basis}
The set $\{R_{ij}: 1 \leq i < j \leq r\}$ is a basis of $\calA(\bL)$.
\end{proposition}
\begin{proof}
The proof is analogous to that of \cite{Eskin-Katznelson1995}*{Proposition~3.3}.
\end{proof}

For $g \in \GL_n(\R)$, let $F_g \colon \M_n(\R)_{\alt} \to \M_n(\R)_{\alt}$ 
be the linear map $A \mapsto gAg^T$. 

\begin{proposition}[Analogue of Proposition~3.4 of \cite{Eskin-Katznelson1995}]
\label{P: scale}
We have $\det F_g = (\det g)^{n-1}$ for all $g \in \GL_n(\R)$.
\end{proposition}

\begin{proof}
Since $g \mapsto \det F_g$ is 
an algebraic homomorphism $\GL_n(\R) \to \R^\times$,
the identity $\det F_g = (\det g)^\alpha$ holds for some $\alpha \in \Z$.
Evaluating at $g=tI$ for any $t \in \R^\times$ yields 
$(t^2)^{n(n-1)/2} = (t^n)^\alpha$, so $\alpha=n-1$.
\end{proof}

\begin{lemma}[Analogue of Lemma~3.5 of \cite{Eskin-Katznelson1995}]
\label{L: changeofbasis}
We have $d(\calA(\bL)) = 2^{r(r-1)/4}d(\bL)^{r-1}$.
\end{lemma}

\begin{proof}
View $\calA(\bL)$ as the lattice generated by the $R_{ij}$ 
(Proposition~\ref{P: basis}).
By Lemma~\ref{L: absbound}\eqref{ijst},
the square of the desired identity is 
a polynomial identity in the $(\ell_i,\ell_j)$.
Thus we may rotate $\ell_1,\ldots,\ell_r$
to vectors in $\R^r$ with the same inner products;
now $n=r$.

The basis of $\calA(\Z^r)$ provided by Proposition~\ref{P: basis}
consists of $r(r-1)/2$ orthogonal vectors of length $\sqrt{2}$,
so for $\bL=\Z^r$, both sides of the identity equal $2^{r(r-1)/4}$.
Now let $\bL$ be any other rank~$r$ lattice in $\Z^r$.
Let $g \in \GL_r(\R)$ be the linear map taking $\Z^r$ to $\bL$.
Then $F_g$ takes $\calA(\Z^r)$ to $\calA(\bL)$.
Replacing $\Z^r$ by $\bL$
scales the two sides of the identity by $\abs{\det F_g}$
and $\abs{\det g}^{r-1}$, respectively;
these factors are equal by Proposition~\ref{P: scale},
so the identity is preserved.
\end{proof}

Combining Lemma~\ref{L: changeofbasis},
Theorem~\ref{T:Minkowski}, 
and Lemma~\ref{L: absbound}\eqref{|R_ij|} yields
\begin{equation}\label{E:goodbasis}
d(\calA(\bL))=2^{r(r-1)/4}d(\bL)^{r-1} \gg_{r} \prod_i|\ell_i|^{r-1} \gg_{r} \prod_{i<j} |R_{ij}|.
\end{equation}

\subsection{Upper bound on $N_{n,r}(T)$}

\begin{lemma}
\label{L:bijection}
The map
\begin{align*}
	\M_r(\Z)_{\alt} &\to \calA(\bL) \\
	(a_{ij}) &\mapsto \sum_{i<j} a_{ij} R_{ij}
\end{align*}
is a bijection that preserves ranks of matrices.
\end{lemma}

\begin{proof}
It is a bijection by Proposition~\ref{P: basis}.
If $\ell_1,\ldots,\ell_r$ are the first $r$ standard basis vectors,
then the bijection simply extends a matrix in $\M_r(\Z)_{\alt}$
by zeros to a matrix in $\M_n(\Z)_{\alt}$; this preserves rank.
The general case follows:
if for some $g \in \GL_n(\R)$
we replace $\ell_1,\ldots,\ell_r$ 
by $g\ell_1,\ldots,g\ell_r$,
then $A \colonequals \sum_{i<j} a_{ij} R_{ij}$ is replaced by $gAg^T$, 
which has the same rank as $A$.
\end{proof}

\begin{lemma}
\label{L:linear dependence}
Let $B = (a_{ij}) \in \M_r(\R)_{\alt}$.
If there exist $i \le j$ with $i+j \le r+1$ such that for every pair $s<t$
with $s \ge i$ and $t \ge j$, we have $a_{st}=0$,
then $\rank B \le r-1$.
\end{lemma}

\begin{proof}
The $i$th through the $n$th row
are contained in the initial copy of $\R^{j-1}$ in $\R^r$,
so they together with the first $i-1$ rows
span a space of dimension at most $(i-1)+(j-1) \le r-1$.
Thus $\rank B \le r-1$.
\end{proof}

Recall that $N_1(T, \calA(\bL))$ counts the matrices in $\calA(\bL)$ of rank exactly $r = \rk \bL$.

\begin{lemma}[Analogue of Lemma~4.1 of \cite{Eskin-Katznelson1995}]
\label{L:getpoint}
If $N_1(T, \calA(\bL)) > 0$, then $|\ell_i||\ell_j| \ll_r T$
for all pairs $(i,j)$ such that $i+j \le r+1$.
\end{lemma}

\begin{proof}
Suppose that $N_1(T,\calA(\bL))>0$.
Thus there exists $A \in \calA(\bL)$ of rank~$r$ with $|A| \le T$.
By Lemma~\ref{L:bijection},
we may write $A = \sum_{i < j} a_{ij} R_{ij}$
for some $B=(a_{ij}) \in \M_r(\Z)_{\alt}$ also of rank~$r$.
Lemma~\ref{L:linear dependence}
yields a pair $s<t$ with $s \ge i$ and $t \ge j$
such that $a_{st} \ne 0$.
By \eqref{E:goodbasis} and Lemma~\ref{L: triangle},
$|A|^2 \gg_r a_{st}^2 |R_{st}|^2 \ge |R_{st}|^2$.
Thus $|\ell_i| |\ell_j| \le |\ell_s| |\ell_t| \asymp_r |R_{st}| \ll_r |A| \le T$
(the second step is Lemma~\ref{L: absbound}\eqref{|R_ij|}).
\end{proof}

\begin{corollary}[Analogue of Corollary~4.2 of \cite{Eskin-Katznelson1995}]
\label{C: suffcondition}
If $N_1(T, \calA(\bL)) > 0$, then $d(\bL) \ll_r T^{r/2}$.
\end{corollary}

\begin{proof}
By Theorem~\ref{T:Minkowski} (at the left) and Lemma~\ref{L:getpoint} (at the right),
\[
	d(\bL)^2 
	\asymp_r (|\ell_1|\cdots |\ell_r|)^2  
	= \prod_{i+j = r+1} |\ell_i| |\ell_j|  
	\ll_r T^r.\qedhere
\]
\end{proof}

\begin{theorem}[Analogue of Theorem~4.1 of \cite{Eskin-Katznelson1995}]
\label{T:mainupperbound}
If $0 \le r \le n-1$ and $r$ is even,
then
\[
	N_{n,r}(T) \ll_n T^{nr/2}.
\]
\end{theorem}

\begin{proof}
Since $N_{n,0}(T)=1$, we may assume that $r \ge 2$.
Let $E$ be the set of primitive rank~$r$ lattices $\bL \subset \Z^n$ 
for which $N_1(T, \calA(\bL)) > 0$.
Let $c_r$ be the implied constant in Corollary~\ref{C: suffcondition},
and let $E^*$ be the set of primitive rank~$r$ sublattices of $\Z^n$ 
for which $d(\bL) \leq c_r T^{r/2}$.
Thus $E \subseteq E^*$.
By~\eqref{E:countbyL},
\[
	N_{n,r}(T) = \sum_{\bL \in E} N_1(T,\calA(\bL)) 
	\le \sum_{\bL \in E} N(T,\calA(\bL)).
\]
By Lemmas~\ref{L:countinlattice} and \ref{L: changeofbasis},
\begin{equation}\label{E:count mat}
	N(T, \calA(\bL)) 
	= \frac{\calV_{r(r-1)/2} T^{r(r-1)/2}}{2^{r(r-1)/4}d(\bL)^{r-1}} +
	 O_r\left(\sum_{k=0}^{r(r-1)/2-1} \frac{T^k}{L_1 L_2 \cdots L_k}\right),
\end{equation}
where the $L_j$ are a reduced basis for $\calA(\bL)$.  
By~\eqref{E:goodbasis} and Lemma~\ref{L: uniformity}, 
we can order the tuples $(i,j)$ with $1\leq i<j \leq r$ 
as $(i_1,j_1),\dots$ so that
$L_k \asymp_r R_{i_k,j_k} \asymp_r |\ell_{i_k}| |\ell_{j_k}|$,
by Lemma~\ref{L: absbound}\eqref{|R_ij|},
where the $l_i$ are a reduced basis for $\bL$.  
Thus 
\begin{equation}
\label{E:error product}
	\sum_{k=0}^{r(r-1)/2-1} \frac{T^k}{L_1 L_2 \cdots L_k}
	\ll_r \prod_{\substack{1\leq i< j\leq r\\
	(i,j)\ne (r-1,r)}} \max(1,\frac{T}{|l_i||l_j|}).
\end{equation}
Let $\calE_1$ be the sum over $\bL\in E$ 
of the right side of~\eqref{E:error product}.  
If $4 \le r \le n-1$,
then the proof of \cite{Eskin-Katznelson1995}*{Proposition~7.1}
yields $\calE_1\ll_n T^{nr/2}$
(in \cite{Eskin-Katznelson1995}*{Proposition~7.1}, the analogous sum is 
instead over $\bL$ for which there is a \emph{symmetric} rank~$r$ matrix 
of size $<T$ with rows in $\bL$,
but the only property used of the $\bL$ there
is the conclusion of Lemma~\ref{L:getpoint}; 
also each summand in our $\calE_1$ is at most
the corresponding summand in \cite{Eskin-Katznelson1995}*{Proposition~7.1}).  
If $r=2$, then the right side of~\eqref{E:error product} is~$1$,
so $\calE_1 = \# E \le \# E^* \ll_n T^n$, 
by Schmidt's theorem (Theorem~\ref{T:Schmidt}).
Thus $\calE_1 \ll_n T^{nr/2}$ in all cases.

It follows that 
\begin{equation}\label{E:upperboundsecondtolast}
N_{n,r}(T) \ll_n T^{r(r-1)/2} \sum_{\bL \in E} d(\bL)^{-(r-1)} + T^{nr/2}.
\end{equation}
We have
\begin{equation}
\label{E:sum over dyadic intervals}
	\sum_{\bL \in E} d(\bL)^{-(r-1)} 
	\le \sum_{\bL \in E^*} d(\bL)^{-(r-1)}
	= \sum_{d(\bL) \le c_r T^{r/2}} d(\bL)^{-(r-1)}
	\ll_r T^{(n-r+1)r/2}
\end{equation}
by summing over dyadic intervals
since for each $t \in \R$,
\[
	\sum_{d(\bL) \in (t/2,t]} d(\bL)^{-(r-1)}
	\ll_r P_{n,r}(t) t^{-(r-1)}
	\ll_r t^{n-r+1},
\]
by Theorem~\ref{T:Schmidt}.
Substituting~\eqref{E:sum over dyadic intervals} 
into~\eqref{E:upperboundsecondtolast}
yields $N_{n,r}(T) \ll_n T^{nr/2}$.
\end{proof}

\subsection{Lower bound on $N_{n,r}(T)$}

\begin{theorem}[Analogue of Theorem~4.2 of \cite{Eskin-Katznelson1995}]
\label{T:mainlowerbound}
If $0 \le r \le n-1$ and $r$ is even, then
\[
N_{n,r}(T) \gg_n T^{nr/2}.
\]
\end{theorem}

\begin{proof}
Since $N_{n,0}(T)=1$, we may assume that $r \ge 2$.
Let $c=c(n,r)$ be as in \cite{Eskin-Katznelson1995}*{Proposition~2.6}.  
Suppose that $\bL$ is \defi{$c$-regular}, i.e., has a reduced basis
with $|\ell_1| \ge c d(\bL)^{1/r}$.
By Theorem~\ref{T:Minkowski}, $|\ell_i| \asymp_n d(\bL)^{1/r}$ for all $i$.
By Lemma~\ref{L: absbound}\eqref{|R_ij|}, $|R_{ij}| \asymp_n d(\bL)^{2/r}$ 
for all $i<j$.
By Lemma~\ref{L:bijection},
the matrix $A \colonequals \sum_{s=1}^{r/2} R_{2s-1,2s} \in \calA(\bL)$ 
is of rank~$r$, and $|A| \asymp_n d(\bL)^{2/r}$.
Thus we can fix $\epsilon = \epsilon_{n,r} >0$
so that $d(\bL) \le \epsilon T^{r/2}$ implies $|A| \le T$
and hence $N_1(T,\calA(\bL)) \ge 1$.
By \cite{Eskin-Katznelson1995}*{Proposition~2.6},
the number of $c$-regular $\bL$ with $d(\bL) \le \epsilon T^{r/2}$
is $\asymp_n T^{nr/2}$,
and each contributes at least $1$ to $N_{n,r}(T)$.
\end{proof}

Theorems \ref{T:mainupperbound} and~\ref{T:mainlowerbound}
imply Theorem~\ref{T:counting matrices by l^2},
and hence Theorem~\ref{thm:EskinKatznelsonAlternating}.

%****************************************************************************
\section{Computational evidence}
\label{S:computational evidence}

\subsection{New evidence}

There is a long history of computational investigations
on the ranks of elliptic curves:
see \cites{Brumer-McGuinness1990,Cremona1997,Bektemirov-Mazur-Stein-Watkins2007}
and the references therein.
In this section, we provide some further experimental evidence 
for our conjecture on ranks,
using the computer algebra system \textsf{Magma} \cite{Magma};
specifically, we use its algorithms developed by 
Steve Donnelly and Mark Watkins for computing the $2$-Selmer group,
and by Tim Dokchitser and Mark Watkins for computing the analytic rank.

We sample random elliptic curves at various heights.  
For a height $H$, we compute uniformly random integers $A,B$ with
\[ 
	A \in [-\sqrt[3]{H/4}, \sqrt[3]{H/4}], 
	\quad B \in [-\sqrt{H/27}, \sqrt{H/27}],
\]
and keep the pair $(A,B)$ if 
$y^2=x^3+Ax+B$ defines an elliptic curve in $\EE_{\leq H} - \EE_{\leq H/2}$.
By this procedure, we generate a uniformly random elliptic curve 
$E \in \EE_{\leq H} - \EE_{\leq H/2}$.  

Next, we compute the $2$-Selmer group $\Sel_2 E$ 
assuming the Riemann hypothesis for Dedekind zeta functions
to speed up the computation of the relevant class groups and unit groups.  
We also compute $E(\Q)_{\textup{tors}}$.  
Let $r \colonequals \rk_2 \Sel_2 E - \rk_2 E(\Q)_{\textup{tors}}$.
If $r \leq 1$ (and $\Sha(E)$ is finite), 
then $\Sha(E)[2]=\{0\}$ and $r=\rk E(\Q)$.  
Otherwise, we attempt to compute the analytic rank 
$\ord_{s=1} L(E,s)$, which equals $\rk E(\Q)$
(assuming the Birch and Swinnerton-Dyer conjecture).
Computing the analytic rank is the most computationally intensive step.  

The table below displays our results.
For each $H$ shown, $N$ denotes the number of elliptic curves 
generated as above, with heights in $(H/2,H]$.
The entries in the next columns show what percentage of these $N$ 
have rank~$0$, have rank~$1$, etc.;
we expect that each of these entries deviates from the true
percentage (what we would see if we averaged over all 
elliptic curves with heights in $(H/2,H]$)
by at most $O(1/\sqrt{N})$, by the central limit theorem,
so we list $1/\sqrt{N}$ in the final column.

\begin{center}
\begin{tabular}{c|c|cccc|l}
$H$ & $N$ & $0$ & $1$ & $\geq 2$ & $\geq 3$ & $1/\sqrt{N}$ \\
\hline 
$10^{10} \rule{0em}{2.2ex}$ & $279090$ & $32.6\%$ & $47.3\%$ & $17.3\%$ & $2.7\%$ & $0.2\%$ \\
$10^{11}$ & $66006$ & $33.4$ & $47.2$ & $16.9$ & $2.5$ & $0.4$ \\
$10^{12}$ & $29844$ & $34.0$ & $47.2$ & $16.4$ & $2.5$ & $0.6$ \\
$10^{13}$ & $20299$ & $34.8$ & $47.0$ & $15.4$ & $2.6$ & $0.7$ \\
$10^{14}$ & $17836$ & $35.0$ & $47.4$ & $15.3$ & $2.2$ & $0.7$ \\
$10^{15}$ & $5028$ & $36.3$ & $46.4$ & $15.1$ & $2.2$ & $1.4$ \\
\end{tabular}
\end{center}
The proportion with rank $\geq 2$ is slowly but steadily decreasing; 
this seems consistent with our model's predicted proportion 
of $H^{-1/24+o(1)}$.  
The data for rank $\geq 3$ also seems consistent with our prediction.
The raw values predicted by our model using the functions in the last row below
are as follows:
\begin{center}
\begin{tabular}{c|cccc}
$H$ & $0$ & $1$ & $\geq 2$ & $\geq 3$ \\
\hline
$10^{10} \rule{0em}{2.2ex}$  & $30.8\%$ & $42.7\%$ & $19.2\%$ & $7.3\%$ \\
$10^{11}$ & $32.6$ & $43.9$ & $17.4$ & $6.0$ \\
$10^{12}$ & $34.2$ & $45.0$ & $15.8$ & $5.0$ \\
$10^{13}$ & $35.6$ & $45.9$ & $14.4$ & $4.1$ \\
$10^{14}$ & $36.9$ & $46.6$ & $13.0$ & $3.4$ \\
$10^{15}$ & $38.1$ & $47.2$ & $11.9$ & $2.8$ \\
$H$  & $\frac12(1-H^{-1/24})$ & $\frac12(1-H^{-1/12})$ &
$\frac12H^{-1/24}$ & $\frac12H^{-1/12}$ \\
\end{tabular}
\end{center}
But these values should not be read too closely, 
since we are ignoring an implicit factor $H^{o(1)}$.

Further computations on the distribution of Selmer groups and
ranks of elliptic curves grouped by order of magnitude of height 
have been carried out recently by
Balakrishnan, Ho, Kaplan, Spicer, Stein, and Weigandt
\cite{Balakrishnan-Ho-Kaplan-Spicer-Stein-Weigandt2016}.
Our more modest statistical sample is in good agreement with theirs.

\subsection{Biases in the counts of higher rank curves depending on the reduction modulo $p$}

Fix a prime $p$ and an elliptic curve $E$ over $\Q$
with good reduction at $p$.
Define $a_p \in \Z$ by $\#E(\F_p)=p+1-a_p$.
Inspired by random matrix theory,
Conrey, Keating, Rubinstein, and Snaith 
\cite{Conrey-Keating-Rubinstein-Snaith2002}*{Conjecture~2}
conjectured that among even quadratic twists $E_d$, 
the ratio of the number of (analytic) rank $\ge 2$ twists 
with $\left( \frac{d}{p} \right) = 1$
to the number of (analytic) rank $\ge 2$ twists 
with $\left( \frac{d}{p} \right) = -1$
tends to $\sqrt{\frac{p+1-a_p}{p+1+a_p}}$.

We can give a different argument for this conjecture,
based on the heuristic that the ``probability'' that $\sqrt{\Sha_0}=0$
should be inversely proportional to the width of the range in which
the integer $\sqrt{\Sha_0}$ lies.
Specifically, when we solve for $\sqrt{\Sha_0}$ in the
Birch and Swinnerton-Dyer conjecture,
the only systematic difference depending on $\left( \frac{d}{p} \right)$
we expect is in the local $L$-factor $L_p(E_d,s)$,
which, for $\left( \frac{d}{p} \right) = \pm 1$,
is $L_p^\pm(s) \colonequals (1 \mp a_p p^{-s} + p^{1-2s})^{-1}$.
Since $L_p^+(1)/L_p^-(1) = \frac{p+1+a_p}{p+1-a_p}$,
the probabilities that $\sqrt{\Sha_0}=0$
should be in the ratio $\sqrt{\frac{p+1-a_p}{p+1+a_p}}$.
The ratio of the counts of rank $\ge 2$ curves
should equal this ratio of probabilities.

The fact that such a rank count bias has been observed experimentally 
\cite{Conrey-Keating-Rubinstein-Snaith2002}
is further evidence that the methodology 
in Sections \ref{S:average Sha}--\ref{S:consequences} 
of basing conclusions on the distribution of $\#\Sha_0$ 
is reasonable.

%****************************************************************************
\section{Further questions}\label{S:further}

The model in Section~\ref{S:random model}
will yield more predictions about ranks and Shafarevich--Tate groups
of elliptic curves if we can prove the corresponding statements
about the functions $(\rk'_E,\Sha'_E)_{E \in \EE}$ 
of random alternating integer matrices.

\subsection{The density of rank~0 elliptic curves whose Shafarevich--Tate group belongs to a specified class}

For any subset $\calG \subseteq \symplectic$, 
Corollary~\ref{C:random variables for many E}\eqref{I:global Sha' distribution for rank 1}(2)
determines the density
$\ProbE(\{E : \Sha'_E \in \calG \} \mid \rk'_E = 1)$ (with probability~$1$).
The analogous problem for $\rk'_E=0$ is of a different nature
since each $G \in \symplectic$ arises with density~$0$ 
(Corollary~\ref{C:random variables for many E}\eqref{I:distribution of global Sha'}),
so it may be unreasonable to expect the density to exist 
for all of the uncountably many $\calG$ in this case,
but one can still ask about specific $\calG$.

When $\calG$ is the set of squares of cyclic groups, 
Corollary~\ref{C:random variables for many E}\eqref{I:probability of global Sha'}(2)
gives the density of $E$ with $\Sha'_E \in \calG$ (with probability~$1$).
In contrast, consider
$\calG \colonequals 
\{ G \in \symplectic : \textup{$\sqrt{\#G}$ is squarefree}\}$.
The squarefree condition can again be checked one $p$ at a time,
but this time the reductions modulo $p$ lie on a codimension~$1$ subscheme
unfortunately,
and the condition instead is about \emph{squarefree values} of the Pfaffian.
This means that in order to apply the Ekedahl sieve,
we would need results on squarefree values of a multivariable polynomial.
There is a well-known heuristic,
that the density of squarefree values
is the product over $p$ of the density of values not divisible by $p^2$,
but making this rigorous in general seems to require the $abc$ conjecture
\cites{Granville1998,Poonen2003-squarefree} 
(and in the multivariable case the results use a nonstandard
way of counting, involving a nonsquare box).
Thus, although it is clear what to predict
for the density of $E$ such that $\sqrt{\#\Sha(E)}$ is squarefree
among all rank~$0$ elliptic curves,
it is not clear that we can prove the corresponding statement about matrices.

\subsection{Asymptotics for rank~0 elliptic curves with specified Shafarevich--Tate group}

It is conjectured that for each $G \in \symplectic$,
the density of rank~$0$ curves $E$ with $\Sha(E) \isom G$ is $0$:
see Remark~\ref{R:large and small Sha} 
and consequence~\eqref{I:consequence rank 0 has large Sha} 
of Section~\ref{S:theoretical evidence}.
But one can ask a more precise question:

\begin{question}
Given $G \in \symplectic$,
what is the asymptotic growth rate of 
$\#\{E \in \EE_{\le H} : \rk E(\Q)=0 \textup{ and } \Sha(E) \isom G \}$
as $H \to \infty$?
\end{question}

The literature contains contradictory conjectures 
even on whether the set of rank~$0$ curves with $\Sha(E) \isom G$ is finite:
\begin{enumerate}[\upshape (a)]
\item
Elkies~\cite{Elkies2002}*{Section~3.2}
(in a family of quadratic twists)
and Hindry~\cite{Hindry2007}*{Conjecture~5.4} 
(in general)
made conjectures implying that
$L(E,1) \gg H^{-o(1)}$ whenever $L(E,1) \ne 0$.
Combining these with Theorem~\ref{T:bound on Sha_0}\eqref{I:Sha_0 and L(E,1)}
would imply conditionally that $\#\Sha(E) \gg H^{1/12-o(1)}$ 
for all rank~$0$ curves, 
so only finitely many rank $0$ curves would have $\Sha(E) \isom G$.
(Hindry, however, no longer believes 
his conjecture: see~\cite{Hindry-Pacheco2016}*{Observations~1.15(b)}.)
\item \label{I:small Sha Watkins}
Watkins \cite{Watkins2008-ExpMath}*{Section~4.5},
on the other hand, 
considers it likely that among elliptic curves with root number $+1$,
the outcome $\Sha_0=1$ is about as common as $\Sha_0=0$ 
(that is, $\rk E(\Q) \ge 2$).
The numerical data 
in \cite{Dabrowski-Jedrzejak-Szymaszkiewicz2016}*{Section~11}
for a family of quadratic twists
supports this guess.
\end{enumerate}
We suspect that \eqref{I:small Sha Watkins} is the truth,
and more precisely that for each $G$, 
\begin{equation} \label{E:Sha = G}
	\#\{E \in \EE_{\le H} : \rk E(\Q)=0 \textup{ and } \Sha(E) \isom G \} 
	\stackrel{?}= H^{19/24+o(1)}.
\end{equation}
D\k{a}browski, J\k{e}drzejak, and Szymaszkiewicz
have formulated an analogous conjecture for their family of quadratic twists 
\cite{Dabrowski-Jedrzejak-Szymaszkiewicz2016}*{Conjecture~8}.

This raises the question of whether the analogue of~\eqref{E:Sha = G}
for $(\rk'_E,\Sha'_E)$ can be proved.
This has not been done, 
but \cite{Duke-Rudnick-Sarnak1993}*{Example~1.7} 
answer a closely related question
by counting $A \in \M_n(\Z)_{\alt,\le X}$
with $\#(\coker A)$ equal to a given integer,
instead of our finer question about $\coker A$ being a given group.

\subsection{The distribution of the normalized size of the Shafarevich--Tate group of rank~0 curves}

\begin{question}
\label{Q:normalized size}
Does the uniform probability measure on the finite set
\[
	\left\{ \# \Sha(E)/H^{1/12} : 
	E \in \EE_{\le H}, \; \rk E(\Q) = 0 \right\}
\]
converge weakly to a limiting distribution on $\R_{\ge 0}$ as $H \to \infty$?
\end{question}

Our model would predict an answer 
if $X(H)$ and $\nnn(H)$ were specified precisely
instead of requiring only $X(H)^{\nnn(H)} = H^{1/12+o(1)}$.
But the limiting distribution, even if it existed, would not be robust:
for example, it would probably change if we sampled integer matrices
from a sphere instead of a box.
We see no reason for favoring any particular shape,
so we view our model as being insufficient for answering
Question~\ref{Q:normalized size}.

\subsection{The distribution of the Shafarevich--Tate groups of curves of rank $r \ge 2$}

For $r \in \{0,1\}$,
Corollary~\ref{C:random variables for many E}\eqref{I:distribution of Sha'} 
determined that the distribution of $\Sha'_E[p^\infty]$
conditioned on $\rk'_E = r$
agrees with Delaunay's predictions for the actual $\Sha(E)[p^\infty]$.
Can we refine the calculations of Section~\ref{S:counting matrices}
to determine the distribution for $r \ge 2$?
As mentioned in Remark~\ref{R: high rank Sha},
it may very well be that the distribution differs 
from Delaunay's prediction for $r \ge 2$.

We can also ask about the whole group $\Sha'_E$
instead of only one $p$-primary part at a time.
Is it true that for each $r \ge 1$ and $G \in \symplectic$,
\[
	\ProbE \left(\{E \in \EE : \Sha'_E \isom G\} \mid \rk'_E = r \right)
	= \prod_p \ProbE \left(\{E \in \EE : 
	\Sha'_E[p^\infty] \isom G[p^\infty] \} \mid \rk'_E = r \right)
	> 0 ?
\]
We proved this for $r=1$: see Corollary~\ref{C:random variables for many E}\eqref{I:global Sha' distribution for rank 1}(1).

\subsection{Higher rank calibration}

\begin{question}
\label{Q:higher calibration}
Would using an upper bound on $\Sha(E)$
for rank~$r$ curves for some fixed $r \ge 1$ instead of $r=0$
have led to the same calibration of $X(H)^{\nnn(H)}$?
\end{question}

The answer to this question is not necessary for our model,
but a positive answer could be viewed as further support for it.
Many of the necessary ingredients are in place:
\begin{enumerate}[\upshape 1.]
\item 
The Riemann hypothesis for $L(E,s)$ implies 
$L^{(r)}(E,1) \ll_r N^{o(1)} \le H^{o(1)}$,
where $N$ is the conductor 
(use the proof of \cite{Conrey-Ghosh2006}*{Theorem~1} to bound 
$L(E,s)$ on a circle of radius $1/\log N$ containing $1$, 
and apply the Cauchy integral formula for $L^{(r)}(E,1)$).
\item 
Lang's conjectural lower bound for the canonical height of 
a non-torsion point \cite{Lang1990}*{Section~7}
implies that the regulator is $\ge H^{o(1)}$.
Hindry and Silverman \cites{Silverman1981, Hindry-Silverman1988}
have made significant progress towards this conjecture.
\item 
Bounds on the Tamagawa numbers, torsion, and real period 
are as in Section~\ref{S:average Sha}.
\item 
Substituting these into the Birch and Swinnerton-Dyer conjecture
yields a conjectural upper bound $\#\Sha(E) \stackrel{?}{\leq} H^{1/12+o(1)}$ 
for all $E$ of rank~$r$.
\item 
On the matrix side, for fixed $n$ and $r$,
all $A \in \M_n(\Z)_{\alt,\le X}$ of corank $\ge r$ satisfy
$\# (\coker A)_{\tors} \ll X^{n-r}$ as $X \to \infty$.
\item 
Matching these upper bounds with $n \to \infty$
would yield $X(H)^{\nnn(H)} = H^{1/12+o(1)}$ as before.
\end{enumerate}
But $\Sha(E)$ for $r \ge 1$ is usually small 
(at least conjecturally, 
as discussed in Remark~\ref{R:large and small Sha}),
as is $\# (\coker A)_{\tors}$ for $r \ge 1$,
so it is unclear whether the rare events of a large $\Sha(E)$ 
or a large $\# (\coker A)_{\tors}$ 
occur frequently enough to make these upper bounds sharp.
Hence the answer to Question~\ref{Q:higher calibration} is not clear.

%****************************************************************************
\section{Generalizing to other global fields}
\label{S:global fields}

Fix a global field $K$.
Let $\EE_K$ be a set of elliptic curves over $K$
representing each isomorphism class once.
Let $B_K \colonequals \limsup_{E \in \EE_K} \rk E(K)$.
Thus $B_K$, if finite, is the smallest integer such that 
$\{E \in \EE_K: \rk E(K) > B_K\}$ is finite.
Section~\ref{S:predictions for ranks} suggests that $20 \le B_\Q \le 21$ 
(in rank $21$, the model could go either way 
depending on the sign of the function implicit in the $o(1)$ in the exponent).  
A naive generalization of our model
(see Section~\ref{S:heuristic for global fields})
would suggest that $20 \le B_K \le 21$ for all global fields $K$.

\subsection{Subfield issues}
\label{S:subfield issues}

But there is a problem.
Some elliptic curves $E$ over $K$ may have extra structure 
that our model did not take into account.  
For example, if $E$ is a base change of a curve
over a subfield $K_0 \subsetneq K$ such that $K/K_0$ is Galois,
then the group $G \colonequals \Gal(K/K_0)$ acts on $E(K)$ and $\Sha(E)$.
The exploitation of such curves $E$ leads to 
the theorems of Section~\ref{S:special number fields},
which show in particular that 
there exist number fields $K$ making $B_K$ arbitrarily large.

This makes it clear that separate models are needed 
to describe such curves.
Analogously, Cohen and Lenstra in \cite{Cohen-Lenstra1984}*{\S9,~III}
observed that class groups of cyclic cubic extensions 
are not just abelian groups but $\Z[\zeta_3]$-modules 
and should be modeled as such.
If $E$ descends to a subfield $K_0$ such that $K/K_0$ is not Galois,
the relevance of the extra structure is not as obvious,
but it still may be that a separate model is needed as it is in the Cohen-Lenstra-Martinet heuristics for class groups of arbitrary fields \cite{Cohen-Martinet1990}.

We will not attempt here to construct a model for every possible situation.
Instead we restrict attention to the set $\EE_K^\circ$ 
consisting of $E \in \EE_K$
such that $E$ is not a base change of a curve from a proper subfield.
Let $B_K^\circ \colonequals \limsup_{E \in \EE_K^\circ} \rk E(K)$.

\subsection{Heuristic for global fields}
\label{S:heuristic for global fields}
If $K$ is a number field, let $S$ be the set of archimedean places.
If $K$ is a global function field, let $S$ be any nonempty finite set
of places.
In asymptotic estimates below, we view $K$ and $S$ as fixed;
e.g., $X \ll 1$ would mean that $X$ is bounded by a constant
depending on $K$ and $S$.
Let $\calO_{K,S}$ be the ring of $S$-integers in $K$.
For simplicity, we assume that $\Cl \calO_{K.S}$ is trivial.
(If one does not want to enlarge $S$ to ensure this,
one can use the finiteness
of $\Cl \calO_{K,S}$ to verify that the estimates remain valid up
to bounded factors.)
If $v$ is a nonarchimedean place of $K$, let $\calO_v \subset K_v$
be the valuation ring.
For each place $v$ of $K$, we fix a Haar measure $\mu_v$ on $K_v$:
if $v$ is archimedean, let $\mu_v$ be Lebesgue measure;
if $v$ is nonarchimedean, choose $\mu_v$ so that $\mu_v(\calO_v)=1$.
For $a \in K_v$, let $|a|_v$ be the factor by which
multiplication-by-$a$ scales $\mu_v$.
For simplicity we assume $\Char K \ne 2,3$ from now on;
minor modifications would be needed to handle the general case.

Each $E \in \EE_K$ is represented by an equation
$y^2=x^3+Ax+B$ with $A,B \in K$ uniquely determined
up to replacing $(A,B)$ by $(\lambda^4 A,\lambda^6 B)$
for $\lambda \in K^\times$.
Choosing $\lambda$ judiciously let us assume that $A,B \in \calO_{K,S}^\times$.
Since $\Cl \calO_{K,S} = \{1\}$,
we may assume also that for every $v \notin S$ we have $v(A)<4$ or $v(B)<6$.
The only remaining freedom is to scale $(A,B)$ 
using a $\lambda$ in $\calO_{K,S}^\times$.
For $v \in S$, 
define $\height_v E = H_v \colonequals \max\{ |4A^3|_v,|27B^2|_v \}$.
By the product formula, 
the \defi{height} $\height E = H \colonequals \prod_{v \in S} H_v \in \R_{>0}$ 
is independent of the scaling.
Our model for $\rk E(K)$ is the same as for $\Q$, 
with $X^n$ to be related to this new $H$.

Let $\Delta \colonequals -16(4A^3+27B^2) \in \calO_{K,S}$.
Let $\omega \colonequals \frac{dx}{2y}$.
Let $\Omega_v \colonequals \int_{A(K_v)} |\omega_v| \mu_v \in \R_{>0}$.
Let $\Omega \colonequals \prod_{v \in S} \Omega_v$.
Essentially the same argument as in Section~\ref{S:real period} shows that
\begin{equation}
\label{E:Omega_v}
	H_v^{-1/12} \ll \Omega_v \ll H_v^{-1/12} \log(H_v/|\Delta|_v).
\end{equation}
By the product formula, $\prod_{v \in S} |\Delta|_v \ge 1$,
so $\sum_{v \in S} \log(H_v/|\Delta|_v) \le \log H$.
By the AM-GM inequality, 
$\prod_{v \in S} \log(H_v/|\Delta|_v) \ll  (\log H)^{\#S}$.
Therefore, taking the product of~\eqref{E:Omega_v} over $v \in S$ yields
$\Omega = H^{-1/12+o(1)}$.

For $v \notin S$, the Tamagawa factors $c_v$ may be bounded in terms
of $v(\Delta)$, so $\prod c_v \le H^{o(1)}$ as before.
By~\cite{Merel1996}, $\#E(K)_{\tors} \ll 1$.

Let $L_S(E,s)$ be the $L$-series of $E$ 
with the Euler factors at $v \in S$ omitted.
We assume that $L_S(E,s)$ admits an analytic continuation to $\C$
that $L_S(E,1) \ge 0$, and that the average of $L_S(E,1)$ 
over $E \in \EE_K$ is $\stackrel{?}\asymp 1$.
Define $\Sha_0$ as before.

The Birch and Swinnerton-Dyer conjecture~\cite{Tate1995}*{Conjecture~(B)} 
implies that 
\[
	L_S(E,1) \stackrel{?}= \delta \frac{\Sha_0 \, \Omega \prod_{v \notin S} c_v}{\#E(K)_{\tors}^2},
\]
where $\delta$ is a constant depending only on $K$
(and our choice of measures $\mu_v$).
As before, we obtain
$\Average_{E \in \EE_{K, \le H}} \Sha_0 \stackrel{?}\asymp H^{1/12+o(1)}$ 
as $H \to \infty$,
which suggests the same calibration $X^{n/2} \sim H^{1/12}$ as for $\Q$.

Next we argue that $\#\EE_{K,\le H} \asymp H^{5/6}$ as $H \to \infty$.
For $v \in S$, choose $C_v \in \R_{>0}$ such that $\prod_{v \in S} C_v = H$.
Parallelotope estimates~\cite{LangAlgebraicNumberTheory}*{V.\S2, Theorem~1} 
imply that the number of $A \in \calO_{K,S}$
satisfying $|4A^3|_v \le C_v$ for all $v \in S$ is $\asymp H^{1/3}$,
and similarly for $B$;
combining these estimates with an elementary sieve 
constructs $\asymp H^{5/6}$ curves,
but some of them may be equivalent 
under scaling by $\lambda \in \calO_{K,S}^\times$.
If we fix suitably small constants $\epsilon_v>0$
and remove the elliptic curves 
satisfying $\max\{|4A^3|_v,|27B^2|_v\} \le \epsilon_v C_v$ for all $v \in S$,
then the equivalence classes of those remaining are of bounded size;
thus $\#\EE_{K,\le H} \gg H^{5/6}$.
On the other hand, for suitably large constants $M_v>0$,
geometry of numbers shows that
every $E \in \EE_{K,\le H}$ is represented by a pair $(A,B) \in \calO_{K,S}^2$
such that $\max\{|4A^3|_v,|27B^2|_v\} \le M_v C_v$ for all $v \in S$,
so parallelotope estimates imply that
$\# \EE_{K,\le H} \ll \prod_v (M_v C_v)^{1/3} (M_v C_v)^{1/2} \ll H^{5/6}$.

\subsection{Number fields}

Let $\EE_{K,\le H}^\circ \colonequals \{E \in \EE_K^\circ : \height E \le H\}$.
If $E \in \EE$ is definable over some $K_0 \subsetneq K$, then $j(E) \in K_0$, 
and this implies nontrivial polynomial equations satisfied by 
the coefficients of $A$ and $B$ relative to a basis for $K$ over $\Q$.
The fraction of $E \in \EE_{K,\le H}$ for which these equations hold
is asymptotically~$0$,
so $\#\EE_{K,\le H}^\circ \asymp \#\EE_{K,\le H} \asymp H^{5/6}$.

Now the same arguments as for $\Q$ suggest $20 \le B_K^{\circ} \le 21$.

\begin{remark}
The upper bound $B_K^{\circ} \le 21$ must fail for many number fields $K$,
however, because of certain special families of elliptic curves,
as we now explain.
Shioda \cite{Shioda1992} proves that 
the elliptic curve $E \colon y^2=x^3+t^{360}+1$ over $\C(t)$ has rank~$68$.
The coefficients involved
in the coordinates of generators of $E(\C(t))$ are algebraic;
let $K$ be any number field containing them all.
Then the rank of $E$ over $K(t)$ is at least $68$.
Next, N\'eron's specialization result \cite{Neron1952}*{IV, Th\'eor\`eme~6}
shows that for $a$ in a density~$1$ subset of $K$,
setting $t=a$ results in an elliptic curve $E_a \in \EE_K^\circ$ 
such that $\rk E_a(K) \ge 68$.
(N\'eron's result states only that infinitely many such $a$ exist,
but its proof, based on the Hilbert irreducibility theorem,
gives a density~$1$ set of such $a$.
In fact, by a refinement of Silverman \cite{Silverman1983}*{Theorem~C}, 
the specialization map is injective with only finitely many exceptions.)
These $E_a$ fall into infinitely many isomorphism classes over $K$,
so $B_K^{\circ} \ge 68$.
\end{remark}

It still seems plausible that $B_K^{\circ}$ and $B_K$ are finite 
for each number field $K$.

\subsection{Number fields with infinitely many elliptic curves of high rank}
\label{S:special number fields}

Using the results of Cornut \cite{Cornut2002} and Vatsal \cite{Vatsal2003}
on Heegner points over 
anticyclotomic $\Z_p$-extensions of imaginary quadratic fields,
one can prove the following.

\begin{theorem}
\label{T:anticyclotomic}
There exists an effectively constructible sequence of number fields $K$
in which $[K:\Q] \to \infty$ and $B_K \ge [K:\Q]/2$.
\end{theorem}

\begin{proof}
Fix a prime $p \ge 5$.
Choose an elliptic curve over $\F_p$
whose trace of Frobenius $a_p$ is not $0$, $1$, or $2$
(see \cite{Mazur1984}*{p.~203} for the relevance of this condition),
and lift it to an elliptic curve $E$ over $\Q$ with good reduction at $p$.
Let $N$ be the conductor of $E$.
By quadratic reciprocity and the Chinese remainder theorem,
we can find an imaginary quadratic field $K_0$ satisfying the
\emph{Heegner hypothesis} that all prime factors of $N$ split in $K_0$.
Let $K_\infty$ be the anticyclotomic $\Z_p$-extension of $K_0$,
let $K_n$ be the degree~$p^n$ subextension,
and let $\Lambda \colonequals \Z_p[[\Gal(K_\infty/K)]]$
be the Iwasawa algebra \cite{Mazur1984}*{Section~17}.
The theorem on page~496 of~\cite{Cornut2002}
implies Mazur's conjecture \cite{Mazur1984}*{bottom of p.~203}
that the Heegner module $\EE(K_\infty)$ \cite{Mazur1984}*{p.~203}
is nonzero, and hence is a free $\Lambda$-module of rank~$1$.
This implies that $\rk E(K_n) \ge p^n$.
(The idea that Heegner points might yield unbounded rank
in anticyclotomic towers is due to Kur\v{c}anov~\cite{Kurcanov1977}.)
For any prime $\ell$ splitting in $K_0/\Q$ 
such that $\left(\frac{\ell}{p}\right)=+1$,
the twist $E_\ell$ satisfies the Heegner hypothesis
and its reduction mod $p$ has the same $a_p$,
so $\rk E_\ell(K_n) \ge p^n$ too.
The base extensions to $K_n$ of these twists
cover infinitely many $K_n$-isomorphism classes,
so $B_{K_n} \ge p^n = [K_n:\Q]/2$.
\end{proof}

Call a number field \defi{multiquadratic} 
if it is a compositum of quadratic fields.
We can obtain a faster rate of growth than in Theorem~\ref{T:anticyclotomic}
by using multiquadratic fields instead of anticyclotomic fields:

\begin{theorem}
\label{T:multiquadratic}
For every $n \ge 0$,
there exists a degree~$2^n$ multiquadratic field $K$
such that a positive proportion of $E \in \EE = \EE_\Q$
satisfy $\rk E(K) \ge 2^n$ and hence $B_K \ge 2^n = [K:\Q]$.
\end{theorem}

\begin{proof}
For $E \in \EE_\Q$, let $\Delta(E)$ be its minimal discriminant,
and for $d \in \Q^\times$ (or $\Q^\times/\Q^{\times 2}$), 
let $E_d$ be the corresponding quadratic twist of $E$.
If $G$ is a finite subgroup of $\Q^\times/\Q^{\times 2}$,
and $K = \Q(\sqrt{G})$ is the corresponding multiquadratic field,
then $\rk E(K) = \sum_{d \in G} \rk E_d(\Q)$,
as one sees by decomposing the $\Gal(K/\Q)$-representation $E(K) \tensor \Q$.

Given $n$, the multidimensional density Hales--Jewett theorem
of Furstenberg and Katznelson \cite{Furstenberg-Katznelson1991}*{Theorem~2.5} 
(reproved by D.~H.~J.~Polymath~\cite{Polymath2012}*{Theorem~1.6})
shows that if $m$ is sufficiently large, any subset of $\F_2^m$
of density $26\%$ or more 
contains an affine $n$-plane.
(The reason for using $26\%$ will be explained later.)
Fix such an $m$.
Let $q_1,\ldots,q_m$ be primes congruent to $1$ modulo $4$
and to $\pm 2$ modulo $5$.
Let $L \colonequals \Q(\sqrt{q_1},\ldots,\sqrt{q_m})$.

Consider the following conditions on an elliptic curve $E \in \EE$:
\begin{enumerate}[\upshape \phantom{mmm}(a)]
\item \label{I:ordinary at 5} 
$E$ has good ordinary reduction at $5$;
\item \label{I:good at q_i}
$E$ has good reduction at every $q_i$; and
\item \label{I:weird condition}
For every prime $\ell \equiv \pm 1 \pmod{5}$
      we have $\ord_\ell(\Delta(E)) \notin \{5,10,15,\ldots\}$.
\end{enumerate}
Let $S$ be the set of $E \in \EE$ 
satisfying \eqref{I:ordinary at 5} and~\eqref{I:good at q_i};
this is a positive density subset of $\EE$.
Now
\begin{enumerate}[\upshape 1.]
\item 
A method of Wong~\cite{Wong2001} 
(see \cite{Bhargava-Skinner-Zhang-BSD-preprint}*{Theorem~16})
shows that $S$ contains a family $F'$ 
(a finite union of large families, in the sense of
\cite{Bhargava-Skinner-Zhang-BSD-preprint}*{Section~2.3})
of relative density $>55.01\%$ in $S$
in which the root number is equidistributed.
\item 
A squarefree sieve shows that at least $99.9\%$ of the curves
in $F'$ satisfy \eqref{I:weird condition} as well
(cf.~\cite{Bhargava-Skinner-Zhang-BSD-preprint}*{Lemma~19});
these curves are contained in the set $S_1(5)$ 
of~\cite{Bhargava-Skinner-Zhang-BSD-preprint}*{Section~3.1}.
\item 
The arguments of \cite{Bhargava-Skinner-Zhang-BSD-preprint}*{Section~3.4}
show that at least $19/40$ of the curves in this subfamily have rank~$1$.
\end{enumerate}
Thus at least $(0.5501)(0.999)(19/40) > 26\%$ of the curves in $S$
have rank~$1$.

Let $\calD$ be the set of products formed by subsets of $\{q_1,\ldots,q_m\}$.
For each $d \in \calD$, let $S_d \colonequals \{E_d : E \in S\}$.
The same arguments as above show that for each~$d$, 
the subset $S_d$ is of positive density in~$\EE$,
and more than $26\%$ of the curves in $S_d$ have rank~$1$.

Choose $c>0$ small enough that $E \in \EE_{\le cH}$ implies 
$E_d \in \EE_{\le H}$ for all $d \in \calD$.
For each $E \in S \intersect \EE_{\le cH}$,
call $\{E_d : d \in \calD \}$ a \defi{hypercube}.
For a positive fraction (independent of $H$) of hypercubes $\calH$,
at least $26\%$ of the $2^m$ curves in $\calH$ have rank~$1$.
In each such $\calH$,
the choice of $m$ guarantees an affine $n$-plane
consisting entirely of twists of rank $\ge 1$;
any one of these twists $E$
has rank at least $2^n$ over the degree~$2^n$ multiquadratic field 
$K \subseteq L$
corresponding to the orientation of the $n$-plane.
But $L$ has only finitely many such subfields,
so one such $K$ occurs for a positive fraction of hypercubes.
These give a positive fraction of $E \in \EE_{\le H}$ 
for which $\rk E(K) \ge 2^n$.
Thus $B_K \ge 2^n$.
\end{proof}

\begin{remark}
The proof of Theorem~\ref{T:multiquadratic}
produces a finite list of degree~$2^n$ multiquadratic
fields $K$ such that one of them satisfies $B_K \ge 2^n$,
but it seems that we cannot determine effectively which $K$ it is!
We can, however, effectively construct their compositum $L$,
a multiquadratic field of larger degree such that $B_L \ge 2^n$.
\end{remark}

\subsection{Function fields}
\label{S:function fields}

Let $K$ be a global function field.
Tate and Shafarevich \cite{Tate-Shafarevich1967}
and Ulmer \cite{Ulmer2002} constructed families of elliptic curves
showing that $B_K = \infty$.
But the elliptic curves of high rank constructed
are always defined over proper subfields of $K$.
(For example, in \cite{Tate-Shafarevich1967} the curves are isotrivial 
and not necessarily constant, but still they are defined over a proper
\emph{infinite} subfield of $K$.)
Thus $B_K^{\circ}$ may still be finite.

%****************************************************************************
% REFERENCES

\end{document}